\documentclass[11pt,reqno]{amsart}
\usepackage[T1]{fontenc}
\usepackage{lmodern,amsmath,amssymb,amsthm,mathtools,microtype,booktabs,needspace}
\usepackage[margin=1.03in]{geometry}
\usepackage[hidelinks]{hyperref}
\hypersetup{pdftitle={Davenport constants with growing binary rank},
pdfauthor={Chao Liu}}
\allowdisplaybreaks[2]
\newtheorem{theorem}{Theorem}[section]
\newtheorem{proposition}[theorem]{Proposition}
\newtheorem{lemma}[theorem]{Lemma}
\newtheorem{corollary}[theorem]{Corollary}
\theoremstyle{definition}
\newtheorem{definition}[theorem]{Definition}
\newtheorem{example}[theorem]{Example}
\theoremstyle{remark}
\newtheorem{remark}[theorem]{Remark}
\newcommand{\F}{\mathbb F}
\newcommand{\Z}{\mathbb Z}
\newcommand{\D}{\mathsf D}
\newcommand{\Ds}{\mathsf D^*}
\newcommand{\dd}{\mathsf d}
\newcommand{\supp}{\operatorname{supp}}
\newcommand{\Sp}{\operatorname{sp}}
\newcommand{\wt}{\operatorname{wt}}
\newcommand{\ord}{\operatorname{ord}}
\newcommand{\kap}{\kappa}
\newcommand{\rk}{\operatorname{rank}}
\newcommand{\Id}{\mathbf 1}
\newcommand{\Vb}[2]{V_{#1}(#2)}
\newcommand{\ceil}[1]{\left\lceil#1\right\rceil}
\newcommand{\floor}[1]{\left\lfloor#1\right\rfloor}
\newcommand{\ind}{\mathbf 1}
\title[Davenport constants with growing binary rank]
{Davenport constants with growing binary rank}
\author{Chao Liu}
\address{Department of Mathematics, University of Mississippi}
\email{chaoliuac@gmail.com}
\date{Working draft, revised October 28, 2026}
\begin{document}
\begin{abstract}
We study the correction to the standard lower bound $\Ds$ for
Davenport constants of finite abelian groups with a growing elementary
two-component. For every fixed nontrivial abelian $p$-group $H$ of
odd order, we prove
$\D(H\oplus C_2^r)=\Ds(H\oplus C_2^r)
+(\D(H)-1)\log_2r+O_H(1)$, with a uniform version for varying $H$.
For $C_3\oplus C_2^r$, we determine exact values on explicit infinite
intervals of ranks. More generally, the correction is logarithmic or
smaller precisely when the fixed group has at most one odd prime divisor.
We show that the two-primary part affects the logarithmic coefficient:
for $C_{24}\oplus C_2^r$, the ratio $(\D-\Ds)/\log_2r$ has limsup
at least $2.4933231\ldots>2$, whereas for $C_3\oplus C_2^r$ the
limit is $2$. These upper limiting coefficients are unbounded even
with the odd prime fixed. The proofs use binary codes and finite-field
Fourier analysis.
\end{abstract}
\maketitle

\section{Introduction}\label{sec:introduction}
For a finite abelian group $G$, the Davenport constant $\D(G)$ is the
least positive integer $N$ such that every sequence of $N$ elements of
$G$ contains a nonempty zero-sum subsequence. Repeated elements are
allowed. If
\[
 G\cong C_{n_1}\oplus\cdots\oplus C_{n_s},
 \qquad 1<n_1\mid\cdots\mid n_s,
\]
the standard construction gives
\[
 \Ds(G):=1+\sum_{i=1}^s(n_i-1)\le\D(G).
\]
Equality holds for finite abelian $p$-groups by Olson's
theorem~\cite{OlsonI}. We study $H\oplus C_2^r$, with particular
attention to exact values, sharp logarithmic corrections, and the
effect of the fixed group $H$ on those corrections.

\subsection{Sharp logarithmic corrections and exact values}
We first determine the logarithmic correction for every fixed
odd-primary label group.

\begin{theorem}[Fixed odd $p$-groups]\label{thm:dichotomy}
Let $p$ be an odd prime and $H$ a fixed nontrivial finite abelian
$p$-group. Then
\begin{equation}\label{eq:fixedH}
 \D(H\oplus C_2^r)
 =\Ds(H\oplus C_2^r)+(\D(H)-1)\log_2 r+O_H(1).
\end{equation}
\end{theorem}

Theorem~\ref{thm:uniform} makes the estimate uniform in $H$. If
$b=\D(H)-1$ and $r\ge16b$, its remainder relative to
$\Ds(H\oplus C_2^r)+b\log_2(r/b)$ lies in $[-3b,0)$.
This also handles varying groups with $b=o(r)$.

For $C_3\oplus C_2^r$, retaining signed Fourier information yields
exact values on explicit infinite intervals of ranks. For $r\ge1$,
the standard lower bound for this group is $r+5$.

\Needspace{14\baselineskip}
\begin{theorem}[Infinite exact intervals]\label{thm:windows}
For each integer $m\ge2$, put
\[
 a_m=2^{m+1}-2m-2,\qquad
 z_m=\floor{\frac{2^{m+2}}{\sqrt3}}-2m-2.
\]
For every integer $a_m\le r\le z_m$,
\begin{equation}\label{eq:windows}
 \D(C_3\oplus C_2^r)=\Ds(C_3\oplus C_2^r)+2m-4=r+2m+1.
\end{equation}
In particular, for every $m\ge1$,
\begin{equation}\label{eq:special-ranks}
 \D(C_3\oplus C_2^{\,2^{m+1}-2m-2})=2^{m+1}-1.
\end{equation}
The case $m=1$ in the last identity is simply
$\D(C_3)=\Ds(C_3)=3$.
\end{theorem}

The interval lengths are asymptotic to $(4/\sqrt3-2)2^m$ and hence
are proportional to their starting ranks. Section~\ref{sec:C3} gives
integer refinements of both endpoints. The signed Fourier identities
used for the upper bound retain information beyond a sphere-volume
estimate.

Theorem~\ref{thm:exact-transfer} gives a complementary exact identity,
expressing $\D(H\oplus C_2^r)$ through the minimum Fourier support
of an $H$-valued function with a unique zero, for every odd-order
abelian $H$. Appendix~\ref{sec:large-rank} localizes the remaining
$C_3$ extremal parameters to two adjacent dimensions for all
sufficiently large ranks; those two thresholds remain unevaluated.

\subsection{Mixed-characteristic fibers}
A zero-sum-free sequence gives a binary selection code: its binary
projection must vanish, and the remaining $H$-valued sum has a
unique zero on that kernel. When $H$ has a two-primary part, the
two-primary and odd-primary conditions have different natural
representations. The former has bounded algebraic degree over
$\F_2$; the latter has bounded Fourier sparsity in odd characteristic.
The following theorem controls their intersection directly.

\begin{theorem}[Mixed-characteristic intersection]\label{thm:main-intersection}
Let $\mathbb K$ be any field of characteristic different from two.
Suppose $a:\F_2^d\to\F_2$ has algebraic degree at most
$D\in\Z_{\ge0}$, and $b:\F_2^d\to\mathbb K$ has Fourier sparsity
at most $S$, where $S\ge1$. If
\[
 Z=\{x\in\F_2^d:a(x)=1,\ b(x)\ne0\}
\]
is nonempty, then
\begin{equation}\label{eq:main-density}
 |Z|\ge 2^{d-D}S^{-2^D}.
\end{equation}
If $Z$ is a singleton, then
\begin{equation}\label{eq:main-isolation}
 d\le D+2^D\log_2 S.
\end{equation}
Both conclusions remain valid on every binary affine restriction.
\end{theorem}

The proof extracts a fixed-dimensional linear space inside the
nonzero-set of $b$, on which a low-degree binary polynomial cannot
be a point mass. The bound is uniform in the field, including
characteristic zero. It yields explicit estimates for every attainable
fiber of a bounded-arity ordered product into $Q\times P$, where
$Q$ is a two-group and $P$ is an odd $p$-group. Neither group need
be abelian.

For any finite group $K$, write
\[
 t(K)=|\{p:p\text{ is an odd prime and }p\mid |K|\}|.
\]

\begin{theorem}[Logarithmic threshold]\label{thm:main-threshold}
Let $H$ be a fixed finite abelian group. Then
\begin{equation}\label{eq:main-threshold}
 \D(H\oplus C_2^r)=\Ds(H\oplus C_2^r)+O_H(\log(r+2))
 \quad\Longleftrightarrow\quad t(H)\le1.
\end{equation}
More precisely:
\begin{enumerate}
\item If $t(H)=0$, then $\D(H\oplus C_2^r)=\Ds(H\oplus C_2^r)$
for every $r$.
\item If $t(H)=1$, then
$\D(H\oplus C_2^r)=\Ds(H\oplus C_2^r)+\Theta_H(\log r)$.
\item If $t=t(H)\ge2$, there is $c_H>0$ such that, for all sufficiently
large $r$,
\begin{equation}\label{eq:main-superlog}
 \D(H\oplus C_2^r)\ge\Ds(H\oplus C_2^r)
 +c_H\left(\frac{\log r}{\log\log r}\right)^t.
\end{equation}
\end{enumerate}
\end{theorem}

The multi-prime lower mechanism is the Barrington--Beigel--Rudich
construction and its coding interpretation~\cite{BBR,BGG}.
The positive qualitative bound also follows from existing group-program
results through the explicit reduction in Section~\ref{sec:programs}.
Our direct approach supplies constants and a common framework for
the arithmetic and fiber statements.

\subsection{Higher two-primary exponents}
For comparing logarithmic coefficients, it is convenient to use
\[
 E_H(r)=\D(H\oplus C_2^r)-\Ds(H\oplus C_2^r),
 \qquad
 \underline\lambda(H)=\liminf_{r\to\infty}\frac{E_H(r)}{\log_2r},
 \qquad
 \overline\lambda(H)=\limsup_{r\to\infty}\frac{E_H(r)}{\log_2r}.
\]
Theorem~\ref{thm:dichotomy} determines both coefficients for odd
$p$-groups. Higher two-primary exponents can change them.

\begin{theorem}[Two-primary dependence]\label{thm:main-coefficients}
For $r\ge512$,
\[
 \D(C_{24}\oplus C_2^r)
 \le\Ds(C_{24}\oplus C_2^r)+32\log_2r+3.
\]
Moreover,
\begin{equation}\label{eq:main-C24}
 \frac{9}{5H_2(1/5)}
 \le\overline\lambda(C_{24})\le32,
 \qquad
 \frac{9}{5H_2(1/5)}=2.4933231062\ldots>2,
\end{equation}
where $H_2(u)=-u\log_2u-(1-u)\log_2(1-u)$ is binary entropy.
For every fixed odd prime $p$ and integer $a\ge1$,
\begin{equation}\label{eq:main-unbounded}
 \overline\lambda(C_{p2^a})\ge\frac1{H_2(2^{-a})}
 \sim\frac{2^a}{a+\log_2e}\qquad(a\to\infty).
\end{equation}
Thus these lower-bound coefficients are unbounded with $p$ fixed.
\end{theorem}

The lower bound in \eqref{eq:main-C24} holds along an explicit
sequence of ranks. It does not determine the exact Davenport constant,
the optimal coefficient, or the existence of a coefficient limit.
The underlying transfer, Theorem~\ref{thm:box}, converts finite
zero-sum-free configurations in an abelian two-group into entropy
lower bounds with a fixed elementary odd-primary target.
Appendix~\ref{sec:nonabelian} gives corresponding logarithmic-order
conclusions for both small and large nonabelian Davenport constants.

\subsection{Connections with earlier work}
Logarithmic lower bounds for large elementary two-components have
a history in zero-sum theory. For cyclic odd $H$, the lower coefficient
$\D(H)-1$ is already contained in
Liu~\cite[Theorem~4.3 and Corollary~4.4]{Liu}.
The simplex construction in Section~\ref{sec:uniform} gives a direct
form of this non-dispersive block method for arbitrary odd-order
abelian label groups. Coding-theoretic approaches also appear in
Plagne--Schmid~\cite{PS}. Brakensiek, Gopi, and Guruswami~\cite{BGG}
relate modular constraints on binary codes to sparse polynomial and
weak-OR representations. We make the exact passage to Davenport
constants explicit in Section~\ref{sec:transfer}.

The primary algebraic input is Olson's group-ring method. The
finite sphere-covering inequality in Corollary~\ref{cor:sphere}
also follows directly from Han--Zhang~\cite[Proposition~2.4]{HZ},
by taking the Hall subgroup to be $C_2^r$.
We prove the stronger code-level covering-radius constraint and
combine it with matching simplex constructions, uniform estimates,
and signed Fourier identities for the exact $C_3$ intervals.
Section~\ref{subsec:covering} compares direct substitution into the
same criterion for $C_{24}\oplus C_2^r$.

Idziak, Kawa{\l}ek, Krzaczkowski, and Wei{\ss}~\cite{IKKW} establish
Boolean-conjunction lower bounds for programs over certain fixed
groups and deduce polynomial lower bounds on nonempty solution
densities. Section~\ref{sec:programs} constructs an explicit fixed
wreath-product simulation of local maps on affine codes. It places
our fixed nilpotent targets in their setting and recovers the
qualitative fiber estimates and logarithmic Davenport upper bounds.
The quantitative estimates here have independent proofs by Fourier
support and primary group algebras.

\subsection{Organization}
Section~\ref{sec:preliminaries} develops the selection-code and primary
representations. Section~\ref{sec:uniform} proves the sharp odd-primary
and uniform estimates. Sections~\ref{sec:transfer} and~\ref{sec:C3}
give the exact Fourier transfer and the $C_3$ intervals.
Sections~\ref{sec:intersection}--\ref{sec:programs} treat mixed-characteristic
intersections, local fibers, and group programs.
Section~\ref{sec:Davenport} proves the general threshold and the
$C_{24}$ upper bound. Section~\ref{sec:entropy} gives the entropy
construction, and Section~\ref{sec:questions} discusses its limitations
and further questions. The appendices contain the nonabelian transfer
and the large-rank $C_3$ localization.

AI tools developed by OpenAI were used in this work.

\section{Binary codes and primary representations}\label{sec:preliminaries}
Unless specified otherwise, logarithms have base two. Write
\[
 \Vb{b}{L}=\sum_{j=0}^b\binom Lj,
\]
with $\binom Lj=0$ for $j>L$. We set $\D(1)=1$, so that
$\D(1)-1=0$. The algebraic degree of a binary function is the
degree of its unique multilinear polynomial over $\F_2$.
The zero function may be assigned degree $-\infty$.

\subsection{The code correspondence}
\begin{lemma}\label{lem:constant-offset}
For every fixed finite abelian group $H$, there is an integer $c_H$
such that
\[
 \Ds(H\oplus C_2^r)=r+c_H
\]
for all sufficiently large $r$.
\end{lemma}
\begin{proof}
For $r$ sufficiently large, the rank of the two-primary component
of $H\oplus C_2^r$ is at least the rank of every odd-primary
component. Every nontrivial invariant factor is then even.
Adding a further $C_2$ inserts one initial invariant factor equal to
two, leaving all the others fixed. It increases $\Ds$ by one.
\end{proof}

If $h_1,\ldots,h_L\in H$, a weighted sum on a binary code is the map
\[
 \sigma_h(x)=\sum_{i=1}^L x_i h_i.
\]
The bits are interpreted as the integers zero and one; this map is
usually not a homomorphism from the binary additive group to $H$.

\begin{proposition}[Selection kernels and realizations]\label{prop:code}
Let $H$ be a finite abelian group.
\begin{enumerate}
\item If $((h_i,v_i))_{i=1}^L$ is zero-sum-free over $H\oplus C_2^r$,
then the binary code
\[
 V=\{x\in\F_2^L:\textstyle\sum_i x_iv_i=0\}
\]
has dimension at least $L-r$, and $\sigma_h^{-1}(0)\cap V=\{0\}$.
\item Conversely, suppose $V\le\F_2^L$ has dimension $d$ and
$\sigma_h^{-1}(0)\cap V=\{0\}$. A full-rank parity-check matrix
$B$ of $V$ gives a zero-sum-free sequence
$((h_i,B_i))_{i=1}^L$ over $H\oplus C_2^{L-d}$.
For every $r\ge L-d$, padding by independent binary basis elements
gives a zero-sum-free sequence of length $r+d$ over $H\oplus C_2^r$.
\end{enumerate}
\end{proposition}
\begin{proof}
A binary selection vector describes a subsequence, and its two
components sum to zero exactly when it belongs to $V$ and its
$H$-sum vanishes. This proves both correspondences. In the second
part, the appended basis elements use new binary coordinates, so a
zero-sum selection cannot contain any of them.
\end{proof}

For later use, if $T:\F_2^d\to\F_2^R$ is linear, the graph code
$\{(x,Tx):x\in\F_2^d\}$ has parity-check matrix $[\,T\mid I_R\,]$.
It has length $d+R$, dimension $d$, and parity-check rank exactly $R$.

\begin{lemma}[Binary minimum weight]\label{lem:binary-weight}
A nonzero binary polynomial of degree at most $e$ in $d$ variables
has at least $2^{d-e}$ nonzero values.
\end{lemma}
\begin{proof}
The assertion is immediate for a nonzero constant. Otherwise
write $f(x',x_d)=g(x')+x_dh(x')$ with $h\ne0$, after relabeling
a variable on which $f$ depends. The degree of $h$ is at most
$e-1$. For every $x'$ with $h(x')=1$, the two values of $f$
differ, so at least one is nonzero. Induction gives at least
$2^{(d-1)-(e-1)}=2^{d-e}$ such values.
\end{proof}

\subsection{Fourier support in characteristic different from two}
For a field $\mathbb K$ with $\operatorname{char}\mathbb K\ne2$,
put $\chi_u(x)=(-1)^{u\cdot x}$ and
\[
 f(x)=\sum_{u\in\F_2^d}\widehat f(u)\chi_u(x),
 \qquad
 \widehat f(u)=2^{-d}\sum_{x\in\F_2^d}f(x)\chi_u(x),
 \qquad
 \Sp(f)=|\supp\widehat f|.
\]
The characters form a basis over $\mathbb K$. Translation preserves
sparsity. If $x=Az+c$ is an affine substitution, each character
restricts to a scalar multiple of one character of $z$, so sparsity
does not increase. Also
\begin{equation}\label{eq:sparsity-product}
 \Sp(fg)\le\Sp(f)\Sp(g),
\end{equation}
because Fourier support of a product is contained in the sumset of
the two supports.

\begin{lemma}[Finite-field uncertainty]\label{lem:uncertainty}
For every nonzero $f:\F_2^d\to\mathbb K$,
\[
 |\supp f|\,\Sp(f)\ge2^d.
\]
\end{lemma}
\begin{proof}
Let $\mathcal T_f$ be the linear span of all translates of $f$.
Fourier diagonalization shows that its dimension is $\Sp(f)$:
character averaging extracts each frequency with nonzero coefficient.
Choose a basis of translates. Their supports cover $\F_2^d$, since
otherwise every translate of $f$ would vanish at a common point,
forcing $f=0$. Each basis function has support size $|\supp f|$.
The union bound proves the claim. No absolute values or positivity
are used.
\end{proof}

A multilinear polynomial of degree at most $b$ in $L$ binary
coordinates has Fourier support at most $\Vb{b}{L}$.
Indeed $x_i=(1-\chi_{e_i})/2$, so only characters supported on at
most $b$ of the original coordinates occur. The same bound holds
after affine restriction.

\subsection{Primary group algebras}
For a finite $p$-group $R$, write $I_R$ for the augmentation ideal
of $\F_p[R]$ and put
\[
 \kap(R)=\max\{j:I_R^j\ne0\},\qquad \kap(1)=0.
\]
The ideal is nilpotent. One elementary proof is induction on $|R|$:
choose a central element $z$ of order $p$. The kernel of
$\F_p[R]\to\F_p[R/\langle z\rangle]$ is the central ideal generated
by $z-1$, whose $p$th power vanishes. Nilpotence in the quotient
therefore gives nilpotence in $\F_p[R]$. Jennings' theorem provides
the exact general index~\cite{Jennings}.

If $R=\bigoplus_i C_{p^{a_i}}$ is abelian, the truncated-polynomial
description gives
\begin{equation}\label{eq:abelian-kappa}
 \F_p[R]\cong
 \F_p[T_1,\ldots,T_s]/(T_1^{p^{a_1}},\ldots,T_s^{p^{a_s}}),
 \qquad
 \kap(R)=\sum_i(p^{a_i}-1)=\D(R)-1.
\end{equation}

Let $g_1,\ldots,g_T$ be group-valued functions. In the relevant
primary group algebra,
\begin{equation}\label{eq:augmentation}
 \prod_{j=1}^T[g_j(x)]
 =\prod_{j=1}^T\bigl(1+([g_j(x)]-1)\bigr).
\end{equation}
Every term with more than $\kap(R)$ augmentation factors is zero.
The expansion preserves the original order, including for
nonabelian groups. Taking the coefficient of $[h]$ gives the exact
indicator that the group product equals $h$.

\section{Sharp estimates for odd-primary label groups}\label{sec:uniform}
\subsection{A covering-radius constraint}

For a binary linear code $C\subseteq\F_2^L$, its covering radius is
\[
 \rho(C)=\max_{w\in\F_2^L}\min_{c\in C}d_H(w,c),
\]
where $d_H$ is Hamming distance. We use the augmentation-ideal identity \eqref{eq:abelian-kappa}.
In group-algebra formulas below, $X^h$ denotes the basis element $[h]$.

\begin{theorem}[Covering transfer]\label{thm:cover}
Let $H$ be a finite abelian $p$-group with $p$ odd, and put $b=\D(H)-1$.
For a sequence
\[
 S=((h_i,v_i))_{i=1}^L\quad\text{over }H\oplus\F_2^r,
\]
define its binary character code by
\[
 C(S)=\{(\langle x,v_1\rangle,\ldots,\langle x,v_L\rangle):
 x\in\F_2^r\}\subseteq\F_2^L.
\]
If $S$ is zero-sum free, then
\begin{equation}\label{eq:radius}
 \rho(C(S))\le b.
\end{equation}
\end{theorem}
\begin{proof}
Let $I$ be the augmentation ideal of $K[H]$, where $K=\F_p$.
Fourier decomposition on the elementary two-component
gives an injective algebra map
\[
 K[H\oplus\F_2^r]\longrightarrow\prod_{x\in\F_2^r}K[H],\qquad
 X^{(h,v)}\longmapsto\bigl((-1)^{\langle x,v\rangle}X^h\bigr)_x.
\]
Indeed, the character matrix on $\F_2^r$ is invertible because $2$ is
invertible in $K$.

Suppose that a vector $w\in\F_2^L$ has distance at least $b+1$ from
every codeword. Set $\varepsilon_i=(-1)^{w_i+1}$ and consider
\[
 F=\prod_{i=1}^L(1-\varepsilon_iX^{(h_i,v_i)}).
\]
In the component indexed by $x$, the factor at position $i$ belongs to
$I$ whenever $w_i\ne\langle x,v_i\rangle$, since that factor is then
$1-X^{h_i}$. At least $b+1$ factors belong to $I$, so every component
vanishes by \eqref{eq:abelian-kappa}. Hence $F=0$.
On the other hand, if $S$ is zero-sum free, only the empty selection
contributes to the coefficient of $X^0$ in $F$. That coefficient is $1$,
which is impossible.
\end{proof}

\begin{corollary}[Finite sphere-covering bound]\label{cor:sphere}
Every zero-sum-free sequence of length $L$ over $H\oplus C_2^r$ satisfies
\begin{equation}\label{eq:sphere}
 2^{L-r}\le \Vb{b}{L}.
\end{equation}
Consequently,
\begin{equation}\label{eq:finite-upper}
 \D(H\oplus C_2^r)
 \le \min\{L\ge1: \Vb{b}{L}<2^{L-r}\}.
\end{equation}
\end{corollary}
\begin{proof}
Theorem~\ref{thm:cover} proves \eqref{eq:sphere} by covering $\F_2^L$
with at most $2^r$ balls of size $\Vb{b}{L}$. Here is a direct algebraic
verification as well. Let $A:\F_2^L\to\F_2^r$ have columns $v_i$ and
put $d=\dim\ker A\ge L-r$. Parameterize its kernel linearly by
$x\in\F_2^d$, with coordinates $\ell_i(x)=\langle t_i,x\rangle\in\{0,1\}$.
For $\chi_t(x)=(-1)^{\langle t,x\rangle}$, one has
$\ell_i=(1-\chi_{t_i})/2$ in $\F_p$.
Write $[X^0]$ for the coefficient at $X^0$ in $\F_p[H]$.
Zero-sum freeness gives the identity
\[
 \delta_0(x)=[X^0]\prod_{i=1}^L
       \bigl(1+\ell_i(x)(X^{h_i}-1)\bigr).
\]
Indeed, the product is $X^{\sum_i\ell_i(x)h_i}$, and its coefficient
at $X^0$ is one precisely when $x=0$.
Expand the product. Every term involving more than $b$ factors
$X^{h_i}-1$ vanishes by $I^{b+1}=0$. Since
$\ell_i=(1-\chi_{t_i})/2$, the Fourier support of the remaining
function is contained in
\[
 \left\{\sum_{i\in J}t_i:J\subseteq\{1,\ldots,L\},\ |J|\le b\right\},
\]
which has at most $\Vb{b}{L}$ elements. On the other hand, every
Fourier coefficient of $\delta_0$ equals $2^{-d}\ne0$ in $\F_p$.
Consequently $2^{L-r}\le2^d\le \Vb{b}{L}$, as required.
For fixed $r,b$, the exponential term eventually exceeds the polynomial
$\Vb{b}{L}$, so the minimum in \eqref{eq:finite-upper} exists.
\end{proof}

More precise lower bounds on the size of a covering code can be inserted
in the same way. The transfer uses the covering radius of the particular
code $C(S)$, and it allows zero columns and repeated columns in the
projection matrix. Standard covering-code results, including those of
Graham and Sloane~\cite{GS}, are therefore available after their individual
hypotheses have been checked.

For later comparison with the standard lower bound, write an odd-order
group as $H\cong\bigoplus_{j=1}^s C_{m_j}$ in invariant-factor form.
Pairing its cyclic factors with $s$ binary factors gives, for $r\ge s$,
\begin{equation}\label{eq:Dstar-general}
 \Ds(H\oplus C_2^r)=r+1+2\sum_{j=1}^s(m_j-1).
\end{equation}

\subsection{Simplex constructions}
The next construction uses only that multiplication by a power of two is
an automorphism of $H$.

\begin{proposition}[Simplex blocks]\label{prop:simplex}
Let $H$ be a nontrivial finite abelian group of odd order, and put
$b=\D(H)-1$. For any positive integers $m_1,\ldots,m_b$, set
\[
 r_0=\sum_{j=1}^b(2^{m_j}-1-m_j).
\]
For every $r\ge r_0$,
\begin{equation}\label{eq:simplex-general}
 \D(H\oplus C_2^r)\ge r+1+\sum_{j=1}^b m_j.
\end{equation}
In particular, if $r\ge b(2^m-1-m)$, then
\begin{equation}\label{eq:simplex-blocks}
 \D(H\oplus C_2^r)\ge r+bm+1.
\end{equation}
\end{proposition}
\begin{proof}
Choose a zero-sum-free sequence $h_1,\ldots,h_b$ in $H$. Let $\mathcal S_m$
be the binary simplex code whose generator matrix has as columns all
nonzero vectors of $\F_2^m$. Its length is $2^m-1$, its dimension is $m$,
and each nonzero word has weight $2^{m-1}$.

Take the direct sum $\mathcal K=\bigoplus_{j=1}^b\mathcal S_{m_j}$, of
length $N=\sum_j(2^{m_j}-1)$ and dimension $\sum_j m_j$. Choose a
quotient map $A:\F_2^N\to\F_2^{r_0}$ with kernel $\mathcal K$.
For every position in block $j$, use the $H$-label $2^{-(m_j-1)}h_j$,
where the inverse power denotes the inverse automorphism of $H$.
Pair these labels with the columns of $A$.

A selection whose binary sum is zero has its indicator vector in
$\mathcal K$. Each nonzero block contributes exactly $h_j$ to its
$H$-sum. If the selection is nonempty, this is a nonempty subsequence
sum of $h_1,\ldots,h_b$, hence is nonzero. Thus the resulting sequence
is zero-sum free. Appending $r-r_0$ new independent binary basis
vectors, with zero $H$-labels, preserves this property. Its total length
is $N+r-r_0=r+\sum_j m_j$.
\end{proof}

For $r\ge b$, the choice $m=\lfloor\log_2(1+r/b)\rfloor$ in
\eqref{eq:simplex-blocks} gives
\begin{equation}\label{eq:simplex-rounded}
 \D(H\oplus C_2^r)
 \ge r+b\floor{\log_2(1+r/b)}+1
 \ge r+b\log_2(r/b)-b+1.
\end{equation}

\subsection{Uniform estimates}
\begin{theorem}[Uniform second-order estimate]\label{thm:uniform}
Let $H$ be any nontrivial finite abelian $p$-group, where $p$ is odd, and put
$b=\D(H)-1$ and $G=H\oplus C_2^r$. If $r\ge16b$, then
\begin{equation}\label{eq:uniform}
 \D(G)=\Ds(G)+b\log_2(r/b)+\rho_H(r),\qquad -3b\le \rho_H(r)<0.
\end{equation}
More precisely,
\begin{equation}\label{eq:uniform-rounded}
 \D(G)\le\Ds(G)+\ceil{b\log_2(r/b)}-1.
\end{equation}
Consequently, for any sequence of such groups $H_r$, allowing their primes,
ranks, and cyclic orders to vary, if $b_r=\D(H_r)-1=o(r)$, then
\begin{equation}\label{eq:movingH}
 \D(H_r\oplus C_2^r)
 =\Ds(H_r\oplus C_2^r)+(1+o(1))b_r\log_2(r/b_r).
\end{equation}
Moreover, $\D(H_r\oplus C_2^r)\sim\Ds(H_r\oplus C_2^r)\sim r$.
\end{theorem}

\begin{proof}[Proof of Theorem~\ref{thm:uniform}]
Write $H=\bigoplus_{j=1}^s C_{p^{e_j}}$. By \eqref{eq:abelian-kappa},
$b=\sum_j(p^{e_j}-1)\ge2$. Since $s\le b/2\le r$,
\eqref{eq:Dstar-general} becomes
\begin{equation}\label{eq:Dstar-p}
 \Ds(H\oplus C_2^r)=r+2b+1.
\end{equation}
Subtracting this from \eqref{eq:simplex-rounded} proves the lower bound
in \eqref{eq:uniform}.

For the upper bound, put $u=r/b\ge16$ and
\[
 L=\ceil{r+b\log_2u+2b}.
\]
The inequality $b\ge2$ gives
\[
 \frac Lr\le1+\frac{\log_2u+5/2}{u}\le\frac{45}{32}.
\]
For the last inequality, $(\log_2u+5/2)/u$ is decreasing for $u\ge16$
and equals $13/32$ at $u=16$.
For $1\le b\le L$, setting $z=b/L\le1$ yields
\[
 z^b\Vb{b}{L}\le(1+z)^L\le e^b,
 \qquad \Vb{b}{L}\le(eL/b)^b.
\]
Using $e<11/4$, we obtain
\[
 \Vb{b}{L}
 <\left(\frac{495}{128}\,\frac rb\right)^b
 <(4r/b)^b\le2^{L-r}.
\]
Corollary~\ref{cor:sphere} gives $\D(H\oplus C_2^r)\le L$.
The rounding is exact:
\[
 L=r+2b+\ceil{b\log_2(r/b)}
  =\Ds(G)+\ceil{b\log_2(r/b)}-1.
\]
This proves \eqref{eq:uniform-rounded}. Since $\lceil y\rceil-1<y$
for every real $y$, it also gives $\rho_H(r)<0$ in \eqref{eq:uniform}.

For the asymptotic assertion, $u_r=r/b_r\to\infty$ and
$|\rho_{H_r}(r)|\le3b_r=o(b_r\log_2u_r)$, proving \eqref{eq:movingH}.
Furthermore, $b_r\log_2u_r/r=\log_2u_r/u_r\to0$ and
$(2b_r+1)/r\to0$, giving the stated first-order asymptotics.
\end{proof}

\begin{proof}[Proof of Theorem~\ref{thm:dichotomy}]
For fixed $H$, \eqref{eq:uniform} already proves
\eqref{eq:fixedH} with an explicitly bounded remainder. More precisely,
for $r\ge16b$ it gives
\[
 \D(H\oplus C_2^r)
 =\Ds(H\oplus C_2^r)+b\log_2r+R_H(r),
 \qquad -b\log_2b-3b\le R_H(r)<-b\log_2b.
\]
\end{proof}
The factorial in the sphere volume gives a sharper explicit upper bound.

\begin{proposition}[An explicit factorial bound]\label{prop:factorial}
Under the hypotheses of Theorem~\ref{thm:uniform}, assume only $r\ge3b$
and define
\[
 \varepsilon_b(r)=\frac{b(b\log_2r+3)}{r\ln2-b}.
\]
Then
\begin{equation}\label{eq:factorial}
 \D(H\oplus C_2^r)
 \le\Ds(H\oplus C_2^r)
 +\floor{b\log_2r-\log_2(b!)+\varepsilon_b(r)}-2b.
\end{equation}
For fixed $b$, the displayed correction $\varepsilon_b(r)$ tends to zero.
\end{proposition}
\begin{proof}
Put $A=b\log_2r-\log_2(b!)$, $\varepsilon=\varepsilon_b(r)$, and
$L=r+\lfloor A+\varepsilon\rfloor+1$. Let $c=L-r$.
We have $A\ge0$, $L\ge r\ge3b$, and
\[
 A+\varepsilon<c\le b\log_2r+\varepsilon+1.
\]
Since $b!/j!\le b^{b-j}$,
\[
 \Vb{b}{L}\le\frac{L^b}{b!}\sum_{j=0}^b(b/L)^{b-j}
 \le\frac{L^b}{b!(1-b/L)}.
\]
Write $\alpha=b/(r\ln2)<1$. The estimates
$\log_2(1+c/r)\le c/(r\ln2)$ and
$-\log_2(1-b/L)\le2b/(r\ln2)$ give
\[
 \log_2 \Vb{b}{L}
 \le A+\alpha c+2\alpha
 \le A+\alpha(b\log_2r+\varepsilon+3)
 =A+\varepsilon<c.
\]
Here $b/L\le1/3$, so the second logarithmic estimate follows from
$-\ln(1-x)\le2x$ for $0\le x\le1/3$.
Apply Corollary~\ref{cor:sphere}.
\end{proof}

\begin{example}
Let $s_r=\lfloor r/(2\log_2r)\rfloor$ and
\[
 G_r=C_6^{s_r}\oplus C_2^{r-s_r}.
\]
Its odd part is $C_3^{s_r}$, so $b_r=2s_r\sim r/\log_2r$. Therefore
\[
 \D(G_r)=\Ds(G_r)+(1+o(1))\frac{r\log_2\log_2r}{\log_2r},
 \qquad \Ds(G_r)=r+4s_r+1.
\]
This follows from the uniform theorem even though the rank of the odd
part tends to infinity.
\end{example}

\subsection{A uniform bound for $C_{2q}^r$}
The covering argument also applies when the odd and binary ranks grow
together.

\begin{corollary}\label{cor:uniform-ratio}
Let $q$ be an odd prime power, and let $\lambda_q>2(q-1)$ be the
unique solution of
\begin{equation}\label{eq:entropy-root}
 \lambda_q\left[1-H_2\!\left(\frac{q-1}{\lambda_q}\right)\right]=1,
\end{equation}
where $H_2(u)=-u\log_2u-(1-u)\log_2(1-u)$. For every integer
$r\ge1$,
\begin{equation}\label{eq:uniform-ratio}
 \D(C_{2q}^r)\le\floor{\lambda_q r}+1
 <1.4162\,\Ds(C_{2q}^r).
\end{equation}
\end{corollary}
\begin{proof}
Put $b=r(q-1)$ and $\Phi_q(x)=x[1-H_2((q-1)/x)]$. On
$x>2(q-1)$, the derivative
$\Phi_q'(x)=\log_2(2(1-(q-1)/x))$ is positive, so $\Phi_q$
increases from zero to infinity. This proves existence and uniqueness
in \eqref{eq:entropy-root}. Take $L=\floor{\lambda_qr}+1>2b$ and
$z=b/(L-b)$. Then
\[
 \Vb{b}{L}\le z^{-b}(1+z)^L=2^{L H_2(b/L)}<2^{L-r},
\]
where the strict inequality follows from
$L[1-H_2(b/L)]=r\Phi_q(L/r)>r$. Corollary~\ref{cor:sphere}, with
$H=C_q^r$, proves the first bound in \eqref{eq:uniform-ratio}.

For the numerical constant, put $c=7081/5000=1.4162$ and
$a=(q-1)/(2q-1)\in[2/5,1/2)$. Define
$F_c(a)=c[1-H_2(a/c)]-(1-2a)$.
Its derivative is $F_c'(a)=2-\log_2((c-a)/a)>0$, because $c<5a$.
The endpoint inequality $F_c(2/5)>0$ is equivalent to the exact
integer comparison
\[
 2^{14081}5^{6000}5081^{5081}>7081^{7081}.
\]
Thus $\Phi_q(c(2q-1))>1$ and $\lambda_q<c(2q-1)$.
Since $\Ds(C_{2q}^r)=r(2q-1)+1$ and $c>1$, we obtain
$\floor{\lambda_qr}+1\le\lambda_qr+1<c\,\Ds(C_{2q}^r)$.
\end{proof}

\section{An exact Fourier formulation}\label{sec:transfer}

The code correspondence has a converse formulation that does not require
$H$ to be a $p$-group. Let $H$ be a nontrivial finite abelian group of odd
order. For $f:\F_2^d\to H$, define
\[
 \widehat f(t)=2^{-d}\sum_{x\in\F_2^d}(-1)^{\langle t,x\rangle}f(x),
 \qquad
 f(x)=\sum_{t\in\F_2^d}\widehat f(t)(-1)^{\langle t,x\rangle}.
\]
These are identities in the additive group $H$: multiplication by
$2^{-d}$ means applying the inverse of multiplication by $2^d$.

\begin{definition}
For $d\ge1$, let $\mu_H(d)$ be the minimum number of nonzero,
nonconstant Fourier coefficients of a function $f:\F_2^d\to H$ such that
\[
 f(0)=0,\qquad f(x)\ne0\quad(x\ne0).
\]
Set $\mu_H(0)=0$.
\end{definition}

The minimum exists: there are finitely many functions, and the function
taking a fixed nonzero value on $\F_2^d\setminus\{0\}$ is admissible.

\begin{theorem}[Exact transfer]\label{thm:exact-transfer}
For every nontrivial finite abelian odd-order group $H$ and every $r\ge0$,
\begin{equation}\label{eq:exact-transfer}
 \D(H\oplus C_2^r)
 =r+1+\max\{d\ge0:\mu_H(d)-d\le r\}.
\end{equation}
\end{theorem}
\begin{proof}
Take a zero-sum-free sequence $((h_i,v_i))_{i=1}^L$ and let
$A:\F_2^L\to\F_2^r$ have columns $v_i$. Put $d=\dim\ker A$ and
parameterize $\ker A$ injectively by $x\in\F_2^d$. Write its coordinate
functions as $\ell_i(x)=\langle t_i,x\rangle\in\{0,1\}$. Then
\[
 f(x)=\sum_{i=1}^L h_i\ell_i(x)
\]
has zero set exactly $\{0\}$. Since
$\ell_i=(1-(-1)^{\langle t_i,x\rangle})/2$, it has at most $L$
nonzero, nonconstant Fourier coefficients. Thus
\[
 \mu_H(d)-d\le L-d=\rk A\le r,\qquad L\le r+d.
\]
This gives the upper bound in \eqref{eq:exact-transfer}.

Conversely, choose an admissible $f$ with $s=\mu_H(d)$ nonconstant
coefficients and let
$E=\{t\ne0:\widehat f(t)\ne0\}$. Since $f(0)=0$,
\[
 f(x)=\sum_{t\in E}h_t\frac{1-(-1)^{\langle t,x\rangle}}2,
 \qquad h_t=-2\widehat f(t).
\]
The set $E$ spans $\F_2^d$: a nonzero vector orthogonal to all of $E$
would give another zero of $f$. Hence the evaluation map
\[
 x\longmapsto(\langle t,x\rangle)_{t\in E}
\]
is injective. Its image $C_f\subseteq\F_2^s$ has dimension $d$.
Let $v_t$ be the images of the standard basis vectors under the quotient
map $\F_2^s\to\F_2^s/C_f\cong\F_2^{s-d}$.
The sequence $((h_t,v_t))_{t\in E}$ is zero-sum free: a nonempty
selection with binary sum zero is a nonzero evaluation vector of some
$x$, and its $H$-sum equals $f(x)\ne0$.

If $s-d\le r$, append $r-(s-d)$ independent binary basis vectors. This
gives a zero-sum-free sequence of length $r+d$, proving the lower bound.
It also proves that the maximum is finite, since
$r+d\le\D(H\oplus C_2^r)-1\le |H|2^r-1$ for each feasible $d$.
The case $d=0$ uses the $r$ independent binary basis vectors directly.
\end{proof}

For $H=C_3=\F_3$, it is convenient to include the constant Fourier
coefficient. Define
\[
 \nu(d)=\min\{|\supp\widehat f|:f:\F_2^d\to\F_3,
 f^{-1}(0)=\{0\}\}\qquad(d\ge1).
\]
Multiplying $f$ by a character preserves its zero set and translates its
Fourier support. By moving a nonzero coefficient to the constant
direction, one obtains
\begin{equation}\label{eq:mu-nu}
 \mu_{C_3}(d)=\nu(d)-1.
\end{equation}
The reverse inequality follows because the full support has at most one
more element than the nonconstant support. Consequently,
\begin{equation}\label{eq:C3transfer}
 \D(C_3\oplus C_2^r)
 =r+1+\max\bigl(\{0\}\cup
 \{d\ge1:\nu(d)-1-d\le r\}\bigr).
\end{equation}
This is an exact structural formula. Evaluating its extremal parameter
for every dimension is a further problem. Its sparse-representation
language is related to the weak-OR formulation in~\cite[Section~4.3]{BGG}.

\section{Exact intervals for the odd part \texorpdfstring{$C_3$}{C3}}\label{sec:C3}

\subsection{Signed Fourier coefficients}
Let $f:\F_2^d\to\F_3$ have zero set $\{0\}$, and write
\[
 f(x)=\sum_{u\in S}a_u(-1)^{\langle u,x\rangle},
 \qquad a_u\in\{1,-1\},\qquad n=|S|.
\]
The set $S$ includes the constant direction when its coefficient is
nonzero. Let $P$ and $Q$ count coefficients $+1$ and $-1$, respectively.

\begin{lemma}[Signed pair bound]\label{lem:signed-pairs}
The support satisfies
\begin{equation}\label{eq:residue}
 n\equiv1-(-1)^d\pmod3,\qquad P-Q\equiv0\pmod3,
\end{equation}
and
\begin{equation}\label{eq:pair-even}
 \binom n2\ge2^d-1.
\end{equation}
If $d$ is odd, then also
\begin{equation}\label{eq:pair-odd}
 2(2^d-1)\le\binom n2+PQ,
 \qquad
 2^d\le1+\frac{3n^2-2n-(P-Q)^2}{8}.
\end{equation}
\end{lemma}
\begin{proof}
As an $\F_3$-valued function, $f^2=1-\delta_0$. The constant Fourier
coefficient of the left side is $\sum_{u\in S}a_u^2=n$; that of the
right side is $1-2^{-d}=1-(-1)^d$ in $\F_3$. Also,
$f(0)=\sum_u a_u=P-Q=0$ in $\F_3$. This proves
\eqref{eq:residue}.

For each $t\ne0$, let $m_t$ count unordered pairs of distinct elements
of $S$ with sum $t$. The nonconstant coefficients of $f^2=1-\delta_0$
give
\begin{equation}\label{eq:pair-convolution}
 2\sum_{\substack{\{u,v\}\subseteq S\\u+v=t}}a_ua_v=-2^{-d}
 \quad\text{in }\F_3.
\end{equation}
In particular $m_t\ge1$, and summing gives \eqref{eq:pair-even}.
If $d$ is odd, the signed sum in \eqref{eq:pair-convolution} is $-1$.
Thus a direction represented by a unique pair must be represented by
a pair of opposite signs.

Let $U$ be the number of uniquely represented nonzero directions. Then
\[
 \binom n2=\sum_{t\ne0}m_t
 \ge U+2(2^d-1-U),\qquad U\le PQ.
\]
These give the first inequality in \eqref{eq:pair-odd}.
Substituting $PQ=(n^2-(P-Q)^2)/4$ gives the second.
\end{proof}

\begin{corollary}[Odd-dimensional support bound]\label{cor:odd-support}
For every integer $s\ge1$,
\begin{equation}\label{eq:odd-support}
 \nu(2s+1)>\frac4{\sqrt3}\,2^s.
\end{equation}
\end{corollary}
\begin{proof}
Write $n=\nu(2s+1)$. Equation~\eqref{eq:pair-even} gives $n\ge5$.
If $n\le(4/\sqrt3)2^s$, then \eqref{eq:pair-odd} would imply
\[
 2^{2s+1}\le1+\frac{3n^2-2n}{8}
 \le2^{2s+1}+1-\frac n4<2^{2s+1},
\]
a contradiction.
\end{proof}

If a zero-sum-free sequence over $C_3\oplus C_2^r$ has length $L$ and
$L-r\ge d\ge1$, choose a $d$-dimensional subspace of the binary
projection kernel. Its label-sum function is admissible, as in
Theorem~\ref{thm:exact-transfer}, and its full Fourier support has size
at most $L+1$. We will use this even when the full kernel is larger
than $d$.

\subsection{Proof of the explicit intervals}
\begin{proof}[Proof of Theorem~\ref{thm:windows}]
For the lower bound, take two simplex blocks of dimension $m$ in
Proposition~\ref{prop:simplex}. Their binary quotient dimension is
$a_m=2(2^m-1-m)$. Hence for every $r\ge a_m$,
\[
 \D(C_3\oplus C_2^r)\ge r+2m+1.
\]

Suppose $m\ge2$ and a zero-sum-free sequence of length $L=r+2m+1$
exists with $r\le z_m$. Put $T=2^m$. Restrict its label-sum function
to a $(2m+1)$-dimensional kernel subspace. It has Fourier support size
$n\le L+1\le4T/\sqrt3$, contradicting
Corollary~\ref{cor:odd-support}. This proves the upper bound. The case $m=1$ is $r=0$ and follows
from $\D(C_3)=3$.
\end{proof}

The isolated values \eqref{eq:special-ranks} also have a short proof
using only the sphere volume. At $L=2^{m+1}-1$,
\[
 \Vb{2}{L}=2^{2m+1}-2^m+1<2^{2m+1}=2^{L-a_m}.
\]
This gives the upper bound directly from Corollary~\ref{cor:sphere}.

\subsection{An integer refinement}
For odd $d\ge1$, let $\lambda_d$ be the least positive integer $n$
satisfying
\begin{equation}\label{eq:lambda}
 n\equiv2\pmod3,\qquad
 3n^2-2n-9\ind_{\{n\text{ odd}\}}\ge8(2^d-1).
\end{equation}
This is an explicit arithmetic bound, not an unknown Fourier extremum.
For instance, set
\[
 R_e(d)=\frac{1+\sqrt{24\cdot2^d-23}}3,\qquad
 R_o(d)=\frac{1+\sqrt{24\cdot2^d+4}}3.
\]
Then
\begin{equation}\label{eq:lambda-closed}
 \lambda_d=\min\left\{
  2+6\ceil{\frac{R_e(d)-2}{6}},\,
  5+6\ceil{\frac{R_o(d)-5}{6}}
 \right\}.
\end{equation}
By \eqref{eq:residue}, when $n$ is odd the integer $P-Q$ is odd and
divisible by three, so $(P-Q)^2\ge9$. Lemma~\ref{lem:signed-pairs}
therefore proves $\nu(d)\ge\lambda_d$.

\begin{proposition}[Refined exact intervals]\label{prop:refined}
Put
\[
 \ell_m=2^{m+1}-2m-3+\ind_{\{m\text{ odd}\}},\qquad
 u_m=\lambda_{2m+1}-2m-3.
\]
For every $m\ge1$ and every integer $\ell_m\le r\le u_m$,
\begin{equation}\label{eq:refined-window}
 \D(C_3\oplus C_2^r)=r+2m+1.
\end{equation}
\end{proposition}
\begin{proof}
On $\F_2^m\oplus\F_2^m$, consider
\[
 f(x,y)=(1-\delta_0(x))+(1-\delta_0(y))\quad\text{in }\F_3.
\]
Its only zero is $(0,0)$: elsewhere its value is one or two.
There are $2(2^m-1)$ nonconstant Fourier coefficients. The constant
coefficient is $2-2^{1-m}$, which is zero for even $m$ and nonzero for
odd $m$. Thus the full support has size
\[
 B_{2m}=2^{m+1}-2+\ind_{\{m\text{ odd}\}}.
\]
Translate this support by a character so it contains a nonzero
constant coefficient. Equations~\eqref{eq:mu-nu} and
\eqref{eq:C3transfer} then give the lower bound for
$r\ge B_{2m}-1-2m=\ell_m$.

For the upper bound, a zero-sum-free sequence of length $r+2m+1$
would yield, on a $(2m+1)$-dimensional kernel subspace, an admissible
function of support size at most
\[
 r+2m+2\le\lambda_{2m+1}-1.
\]
This contradicts $\nu(2m+1)\ge\lambda_{2m+1}$.
\end{proof}

\section{A mixed-characteristic intersection theorem}\label{sec:intersection}
We first extract a linear space from the support of a Fourier-sparse
function. The dimension required for this extraction will depend
only on the degree of the binary condition.

\begin{lemma}[A flat in a sparse support]\label{lem:flat}
Suppose $b:\F_2^d\to\mathbb K$ satisfies $b(0)\ne0$ and
$\Sp(b)\le S$. If an integer $t\ge0$ satisfies
\begin{equation}\label{eq:flat-condition}
 d>t+2^t\log S,
\end{equation}
then $\supp b$ contains a linear subspace of dimension $t+1$.
\end{lemma}
\begin{proof}
Start from $U=\{0\}$. Suppose a $j$-dimensional subspace
$U\subseteq\supp b$ has been constructed. Define
\[
 H_U(x)=\prod_{u\in U}b(x+u).
\]
It is nonzero at zero, and \eqref{eq:sparsity-product} gives
$\Sp(H_U)\le S^{2^j}$. Hence
\[
 |\supp H_U|\ge 2^dS^{-2^j}
\]
by Lemma~\ref{lem:uncertainty}. If $d>j+2^j\log S$, this support
is larger than $|U|$. Choose $v\in\supp H_U\setminus U$.
Then $v+U\subseteq\supp b$, and $U\cup(v+U)$ is a
$(j+1)$-dimensional linear subspace in $\supp b$.
The function $j+2^j\log S$ is nondecreasing, so
\eqref{eq:flat-condition} permits all the required extensions.
\end{proof}

\begin{proof}[Proof of Theorem~\ref{thm:main-intersection}]
First suppose $Z$ is a singleton. Translate its point to zero.
If \eqref{eq:main-isolation} failed, Lemma~\ref{lem:flat} with
$t=D$ would give a $(D+1)$-dimensional subspace
$U\subseteq\supp b$. On $U$, the binary function $a$ is one
at zero and zero elsewhere. In basis coordinates it is
$\prod_{i=1}^{D+1}(1+x_i)$, whose degree is $D+1$.
Affine substitution cannot increase degree, a contradiction.

For general nonempty $Z$, translate a point of $Z$ to zero and
choose an inclusion-maximal linear subspace $W$ with
$W\cap Z=\{0\}$. Put $k=\dim W$. The singleton bound on $W$ gives
$k\le D+2^D\log S$. Every coset of $W$ meets $Z$:
a missing coset $v+W$ would make
$W\cup(v+W)$ a larger isolating subspace. Therefore
\[
 |Z|\ge2^{d-k}\ge2^{d-D}S^{-2^D}.
\]
All ingredients are stable under affine restriction, as required.
\end{proof}

\begin{remark}\label{rem:intersection}
The theorem is not obtained by multiplying separate lower bounds for
$\{a=1\}$ and $\{b\ne0\}$; two sets of large density can have a very
small intersection. The argument instead uses the algebraic degree
of one condition and the translation structure of the other.
Section~\ref{subsec:obstruction} shows that the coefficient of
$\log S$ in the singleton bound cannot universally be replaced by
one, even when $D=3$ and $b$ takes only the values zero and one.
\end{remark}
 
\section{Fibers of local group-valued products}\label{sec:fibers}
Let $V\subseteq\F_2^L$ be an affine code of dimension $d$.
Fix scopes $S_j\subseteq[L]$, with $|S_j|\le\beta$, and maps
$g_j:\F_2^{S_j}\to K$. We consider the ordered product
\begin{equation}\label{eq:local-word}
 F(x)=g_1(x|_{S_1})\cdots g_T(x|_{S_T}),\qquad x\in V.
\end{equation}
The arity $\beta\ge1$ is fixed; the scopes may overlap arbitrarily.
The order in \eqref{eq:local-word} is part of the data.

\begin{theorem}[Explicit local fibers]\label{thm:local}
Let $Q$ be a finite two-group and $P$ a finite odd $p$-group,
and take $K=Q\times P$. Define
\[
 D_0=\beta\kap(Q),\qquad
 W=\sum_{j=0}^{\kap(P)}2^{\beta j}\binom Tj.
\]
Every attainable $h\in K$ satisfies
\begin{equation}\label{eq:local-density}
 \frac{|F^{-1}(h)\cap V|}{|V|}\ge2^{-D_0}W^{-2^{D_0}}.
\end{equation}
If a fiber is a singleton, then
\begin{equation}\label{eq:local-isolation}
 d\le D_0+2^{D_0}\log W.
\end{equation}
\end{theorem}
\begin{proof}
Write $h=(h_Q,h_P)$ and parametrize $V$ by $\F_2^d$.
First work in $\F_2[Q]$. Before restriction to $V$, each local
augmentation-valued function $[g_j^Q(x)]-1$ has a multilinear
representation of degree at most $\beta$ with coefficients in $I_Q$.
By \eqref{eq:augmentation}, terms with more than $\kap(Q)$ such
factors vanish. Thus the coefficient of $[h_Q]$ in the ordered
product is a binary polynomial $a$ of degree at most $D_0$.
It is the exact $Q$-target indicator. The degree bound persists
after the affine parametrization over $\F_2$.

Now work in $\F_p[P]$. A surviving term in the augmentation expansion
has its factors indexed by $J\subseteq[T]$ with $|J|\le\kap(P)$.
It depends on at most $\beta|J|$ original binary coordinates.
On $V$, those coordinates are affine linear forms, so any scalar
coefficient of this term has Fourier support in a space of at most
$2^{\beta|J|}$ characters. Summing over $J$ shows that the exact
$P$-target indicator $b$ has $\Sp(b)\le W$. The desired fiber is
precisely $\{a=1,b\ne0\}$. Apply
Theorem~\ref{thm:main-intersection}.
\end{proof}

\begin{corollary}[Search consequences]\label{cor:algorithms}
For fixed $Q,P,\beta$, existence of a target preimage in
\eqref{eq:local-word} can be tested in randomized polynomial time
with one-sided error. A preimage can be found with the same error
guarantee. Deterministic search and emptiness certification take
time $\exp(O_{Q,P,\beta}((\log(T+2))^2))$, in addition to polynomial
input-reading and linear-algebra costs.
\end{corollary}
\begin{proof}
Project $V$ onto the union of the scopes. Its dimension is at most
$\beta T$, and projection has equal-size fibers, preserving the
density of every target fiber. Uniform sampling succeeds with
inverse polynomial probability by \eqref{eq:local-density}.
Every sampled candidate is checked directly; in an empty fiber
the algorithm never outputs a witness.

For deterministic search, identify the projected code with
$\F_2^{d_0}$. Choose a solution $z$ of minimum Hamming weight $k$.
The coordinate subspace on its support contains no other solution:
every other point has fewer than $k$ nonzero coordinates.
Applying \eqref{eq:local-isolation} to that subspace gives
$k=O_{Q,P,\beta}(\log(T+2))$. Enumerating all vectors up to this
weight gives the asserted running time, including when no solution
exists. A projected witness can be lifted to $V$ by linear algebra.
\end{proof}

\subsection{Sharper estimates for weighted abelian sums}
Let $Q$ be abelian of two-power order and $P$ abelian of odd
prime-power order, allowing either group to be trivial. Set
$b_Q=\D(Q)-1$ and $b_P=\D(P)-1$.

\begin{corollary}[Weighted fibers]\label{cor:weighted}
For labels $h_i\in Q\oplus P$ on an affine binary code $V$ of
dimension $d$ and length $L$, every attainable weighted sum has
at least
\begin{equation}\label{eq:weighted-density}
 2^{d-b_Q}\Vb{b_P}{L}^{-2^{b_Q}}
\end{equation}
preimages. In particular, a singleton fiber implies
\begin{equation}\label{eq:weighted-isolation}
 d\le b_Q+2^{b_Q}\log\Vb{b_P}{L}.
\end{equation}
\end{corollary}
\begin{proof}
The $Q$ indicator has binary degree at most $b_Q$. For the odd
part, expand
\[
 \prod_{i=1}^L\bigl(1+x_i([h_i^P]-1)\bigr)
 \quad\text{in }\F_p[P].
\]
Its target coefficient is a multilinear polynomial of degree at
most $b_P$ in the original bits. It therefore has Fourier support
at most $\Vb{b_P}{L}$, also after affine restriction.
Theorem~\ref{thm:main-intersection} gives the claim.
\end{proof}

When $Q=1$, the singleton bound is just
$d\le\log\Vb{b_P}{L}$, and the density bound also follows directly
from Lemma~\ref{lem:uncertainty}. The mixed intersection argument
is what controls a simultaneous nontrivial two-primary condition.

\section{A connection with programs over fixed groups}\label{sec:programs}
This section gives a second, qualitative route to the preceding
fiber bounds. A program over a fixed finite group $G$ is an ordered
list of instructions, each of which chooses one of two group
elements according to one input bit. An accepting subset of $G$
specifies the Boolean output. Constants can be inserted as
instructions with equal choices. Let $\gamma_G(n)$ denote the
minimum program length representing an $n$-variable singleton
indicator, with value infinity if no such program exists.

We use the following consequence of existing results. By
\cite[Theorem 17]{IKKW}, if
\begin{equation}\label{eq:IKKW-class}
 G\le N\times(R\rtimes A),
\end{equation}
where $N$ is finite nilpotent, $R$ is a finite $p$-group, and $A$
is finite abelian, then $\gamma_G(n)\ge2^{c_G n}$ for all
sufficiently large $n$, with $c_G>0$. Their Proposition~9 converts
this into an inverse-polynomial lower bound for every nonempty
accepting set of a program of length $m$.

\begin{proposition}[Fixed wreath-product simulation]\label{prop:wreath}
Fix $K=Q\times P$ and $\beta$ as in Theorem~\ref{thm:local}.
There is a fixed finite group $W$ satisfying \eqref{eq:IKKW-class}
such that, after parametrizing the relevant projection of $V$,
each target condition $F(x)=h$ in \eqref{eq:local-word} is
recognized by a $W$-program of length $O_\beta(T^2)$.
The simulation preserves the fraction of satisfying inputs.
\end{proposition}
\begin{proof}
Let $A=\F_2^\beta$ and use the regular translation action on
$K^A$. Write
\[
 W=K^A\rtimes A,\qquad
 (f,u)(g,v)=(f\,\tau_u g,u+v),\qquad
 (\tau_u g)(a)=g(a+u).
\]
Multiplication in $K^A$ is pointwise and ordered.

Project the code onto the coordinates that occur in a scope, and
parametrize it by $z\in\F_2^{d_0}$, where $d_0\le\beta T$.
Each local scope is an affine map $y_j(z)\in A$; append unused
coordinates when its size is smaller than $\beta$. Regard the
local table as $f_j\in K^A$. The cursor $(1,y_j(z))$ can be
computed with at most $d_0$ bit-dependent translation instructions
and one constant instruction. Consequently
\begin{equation}\label{eq:conjugate}
 (1,y_j(z))(f_j,0)(1,y_j(z))^{-1}
       =(\tau_{y_j(z)}f_j,0)
\end{equation}
is a program of length at most $2d_0+3$. At coordinate $a=0$,
the right-hand side has value $g_j(y_j(z))$.
The product of these conjugates has cursor zero, and its base
coordinate at zero is exactly $F(x)$. Take the accepting subset
\[
 \{(f,0)\in W:f(0)=h\}.
\]
This gives length at most $T(2d_0+3)$ and the desired accepting
condition.

Finally the diagonal cursor map embeds $W$ into
\begin{equation}\label{eq:wreath-embedding}
 (Q^A\rtimes A)\times(P^A\rtimes A).
\end{equation}
The first factor is a two-group and hence nilpotent; the second
is an odd $p$-group extended by the abelian group $A$.
Thus \eqref{eq:IKKW-class} applies. Uniform binary parameters give
uniform points of the projected affine code, and projection has
equal-size fibers, so all accepting fractions are preserved.
\end{proof}

\begin{corollary}[Qualitative density from group programs]\label{cor:IKKW-density}
For every fixed $Q,P,\beta$ there is $C=C(Q,P,\beta)$ such that
every nonempty fiber in Theorem~\ref{thm:local} has density at least
$(T+2)^{-C}$.
\end{corollary}
\begin{proof}
Apply the program density consequence above to
Proposition~\ref{prop:wreath}. Its length is polynomial in $T$,
and its group is fixed. The resulting density is the original
fiber density.
\end{proof}

The explicit estimate \eqref{eq:local-density} and this reduction
serve different quantitative purposes. The former gives an exponent
in terms of primary augmentation indices and applies a field-uniform
intersection theorem. The latter explains why the qualitative
phenomenon is compatible with fixed-group satisfiability theory.
Thus the polynomial-density and search consequences can be obtained
through this fixed-group connection, while the direct Fourier
argument also provides the explicit estimates needed below.
 
\section{Logarithmic Davenport corrections}\label{sec:Davenport}
\subsection{One odd prime: explicit upper bounds}
Let $H=Q\oplus P$, where $Q$ is a fixed abelian two-group and
$P$ a fixed abelian odd $p$-group. Put
\[
 b_Q=\D(Q)-1,\qquad b_P=\D(P)-1,
 \qquad A_H(L)=b_Q+2^{b_Q}\log\Vb{b_P}{L}.
\]

\begin{proposition}[A finite-rank bound]\label{prop:D-upper}
If $r\ge1$ and $\lfloor A_H(2r)\rfloor\le r-1$, then
\begin{equation}\label{eq:D-upper-finite}
 \D(H\oplus C_2^r)\le r+1+\lfloor A_H(2r)\rfloor.
\end{equation}
Consequently
\begin{equation}\label{eq:D-upper-asymptotic}
 \D(H\oplus C_2^r)
 \le r+2^{b_Q}b_P\log r+O_H(1).
\end{equation}
\end{proposition}
\begin{proof}
If a zero-sum-free sequence has length $L$, its selection kernel
in Proposition~\ref{prop:code} has dimension at least $L-r$.
The zero fiber of its $H$-sum is a singleton. Thus
\begin{equation}\label{eq:length-obstruction}
 L-r\le A_H(L)
\end{equation}
by Corollary~\ref{cor:weighted}. Take
$L=r+1+\lfloor A_H(2r)\rfloor\le2r$. Then
$L-r=1+\lfloor A_H(2r)\rfloor>A_H(2r)\ge A_H(L)$,
contradicting \eqref{eq:length-obstruction}.
This proves \eqref{eq:D-upper-finite}.

For fixed $b_P>0$,
$\log\Vb{b_P}{2r}=b_P\log r+O_H(1)$, so the finite condition
holds for all sufficiently large $r$ and yields
\eqref{eq:D-upper-asymptotic}. If $b_P=0$, the same conclusion
holds with a bounded correction, and Olson's formula gives the
exact value $\D(H\oplus C_2^r)=r+\D(H)$.
\end{proof}

\begin{remark}
Corollary~\ref{cor:IKKW-density} gives the qualitative conclusion
of Proposition~\ref{prop:D-upper} by the same selection-kernel
argument. If its unique fiber has dimension $d$, its density
$2^{-d}$ must be at least $L^{-O_H(1)}$. Hence
$L-r\le d=O_H(\log L)$, which implies $L=r+O_H(\log r)$.
The elementary proof above identifies an explicit coefficient.
\end{remark}

\subsection{Lower bounds and the odd-prime threshold}
We first record the standard simplex-block argument in the form
needed for the logarithmic lower bound.

\begin{corollary}[Logarithmic simplex lower bound]\label{lem:simplex}
For every nontrivial fixed finite abelian group $H$ of odd order,
\[
 \D(H\oplus C_2^r)\ge r+(\D(H)-1)\log r-O_H(1).
\]
In particular, for every odd prime $p$,
\begin{equation}\label{eq:simplex}
 \D(C_p\oplus C_2^r)\ge r+(p-1)\log r-O_p(1).
\end{equation}
\end{corollary}
\begin{proof}
Apply \eqref{eq:simplex-rounded} with $b=\D(H)-1$ fixed.
\end{proof}

If an odd prime $p$ divides $|H|$, then $C_p\le H$, so
\eqref{eq:simplex} and Lemma~\ref{lem:constant-offset} give
$E_H(r)\gg_H\log r$.

\begin{lemma}[Composite-modulus construction]\label{lem:composite}
Let $H$ be a fixed finite abelian group with $t=t(H)\ge2$.
Then \eqref{eq:main-superlog} holds for all sufficiently large $r$.
\end{lemma}
\begin{proof}
Choose the distinct odd primes $p_1,\ldots,p_t$ dividing $|H|$
and set $M=\prod_i p_i$. For a large integer $B$, let $Q_i$ be
the largest power of $p_i$ not exceeding $B+1$, and set
$d=\prod_iQ_i-1$. Thus $d\gg_H B^t$.

For $q=p^k$, the polynomial
\begin{equation}\label{eq:digit-polynomial}
 u_q(x)=\sum_{j=0}^{q-1}(-1)^j e_j(x)\quad\text{over }\F_p
\end{equation}
is the indicator of $q\mid\wt(x)$ on any binary cube; $e_j$
denotes the elementary symmetric polynomial and $e_0=1$.
To verify this, at weight $w\ge1$ use
\[
 \sum_{j=0}^{q-1}(-1)^j\binom wj
       =(-1)^{q-1}\binom{w-1}{q-1}.
\]
Since $q$ is odd, the sign is positive, and Lucas's theorem makes
the last binomial coefficient equal to one modulo $p$ precisely
when $q\mid w$, and zero otherwise. At $w=0$, the sum is one.

By the Chinese remainder theorem, combine the polynomials
$1-u_{Q_i}$ into one polynomial $f$ over $\Z/M\Z$ of degree
at most $B$. Its zero set on $\F_2^d$ is exactly $\{0\}$:
the simultaneous conditions $Q_i\mid\wt(x)$ force
$\prod_iQ_i\mid\wt(x)$, while $0\le\wt(x)\le d$.
This is the symmetric composite-modulus mechanism
of~\cite{BBR}, used in the sparse-polynomial/code setting
in~\cite{BGG}.

Because two is invertible modulo $M$, expand $f$ in parity
characters. At most $\Vb{B}{d}$ frequencies occur.
The identity $f(0)=0$ gives
\[
 f(x)=\sum_{u\ne0}c_u(\chi_u(x)-1)
      =\sum_{u\ne0}(-2c_u)\ell_u(x),
\]
where $\ell_u(x)=u\cdot x\pmod2$ is regarded as a bit.
Use one coordinate for each nonzero coefficient $c_u$.
The resulting linear map to a binary code is injective:
a vector in its kernel would also be a zero of $f$.
The code has dimension $d$ and length
$L\le\Vb{B}{d}-1$, with a unique zero for its weighted $C_M$-sum.

Now $\log\Vb{B}{d}=O_H(B\log B)$.
For each sufficiently large $r$, choose
$B=\lfloor\varepsilon_H\log r/\log\log r\rfloor$, where
$\varepsilon_H>0$ is a sufficiently small fixed constant.
Then $L\le r$, and Proposition~\ref{prop:code} gives a zero-sum-free
sequence of length $r+d$ in $C_M\oplus C_2^r$.
The subgroup $C_M\le H$ and
$\Ds(H\oplus C_2^r)=r+O_H(1)$ finish the proof.
\end{proof}

The same construction has a finite form, useful without an
unspecified asymptotic constant.

\begin{proposition}[Finite composite-modulus bound]\label{prop:CRT-finite}
Let $M=p_1\cdots p_t$ be a product of distinct odd primes, with $t\ge1$. Choose
$Q_i=p_i^{a_i}$, with $a_i\in\Z_{\ge0}$, and put
$Q=\prod_iQ_i$ and $d_0=\max_i(Q_i-1)$. Then
\begin{equation}\label{eq:CRT-finite}
 \D(C_M\oplus C_2^r)\ge r+Q
 \quad\text{whenever}\quad
 r\ge\Vb{d_0}{Q-1}-Q.
\end{equation}
\end{proposition}
\begin{proof}
Combine the polynomials $1-u_{Q_i}$ in
\eqref{eq:digit-polynomial} by CRT on the cube of dimension $Q-1$.
For $Q_i=1$, set $u_1=1$. The resulting polynomial has a unique
zero and degree at most $d_0$, so
$\mu_{C_M}(Q-1)\le\Vb{d_0}{Q-1}-1$.
Theorem~\ref{thm:exact-transfer} now applies. If $Q=1$, the same
conclusion follows from $\mu_{C_M}(0)=0$.
\end{proof}

\begin{example}
For $M=15$, take $Q_1=3^a$ and $Q_2=5^b$. On the cube of
dimension $3^a5^b-1$, the polynomial is explicitly
\[
 f=10\sum_{j=1}^{3^a-1}(-1)^{j+1}e_j
   +6\sum_{j=1}^{5^b-1}(-1)^{j+1}e_j\pmod {15}.
\]
Its only zero is the origin, and its degree is at most
$\max(3^a,5^b)-1$. Both prime components test the same Boolean
weight, over a range of length $3^a5^b$.
\end{example}

\begin{proposition}[An explicit bound at every large rank]\label{prop:mixed-growth}
Let $H$ be fixed and abelian, with distinct odd prime divisors
$p_1,\ldots,p_t$, where $t\ge2$, and put $M=\prod_i p_i$.
For $r\ge4$, set
\[
 R=\log_2r,\qquad
 d=\floor{\frac{R}{(t-1)\log_2R+t+2}},\qquad
 Q_i=p_i^{\lfloor\log_{p_i}(d+1)\rfloor}.
\]
Whenever $d\ge1$,
\begin{equation}\label{eq:mixed-explicit}
 \D(H\oplus C_2^r)\ge r+\prod_iQ_i
 >r+\frac{(d+1)^t}{M}.
\end{equation}
In particular,
\begin{equation}\label{eq:mixed-order}
 \D(H\oplus C_2^r)\ge\Ds(H\oplus C_2^r)
 +\left(\frac1{M(t-1)^t}+o_H(1)\right)
 \left(\frac{\log_2r}{\log_2\log_2r}\right)^t.
\end{equation}
\end{proposition}
\begin{proof}
The group $H$ contains $C_M$. The largest allowed powers satisfy
$(d+1)/p_i<Q_i\le d+1$, whence $Q=\prod_iQ_i>(d+1)^t/M$.
Put $k=Q-1$ and $d_0=\max_i(Q_i-1)\le d$.
If $k\ge d\ge1$, the binomial estimate used in the proof of
Theorem~\ref{thm:uniform} gives
\begin{align*}
 \Vb{d_0}{k}&\le\Vb{d}{k}\le(ek/d)^d\\
 &\le(e2^td^{t-1})^d
 <2^{d((t-1)\log_2R+t+2)}\le2^R=r.
\end{align*}
Here $k<(d+1)^t\le2^td^t$, $d\le R$, and $e<4$.
If $k<d$, use $\Vb{d_0}{k}\le2^k\le2^d\le r$ instead.
Proposition~\ref{prop:CRT-finite} proves \eqref{eq:mixed-explicit}.
Finally $d\sim R/((t-1)\log_2R)$, and
Lemma~\ref{lem:constant-offset} absorbs the fixed difference between
$\Ds(H\oplus C_2^r)$ and $r$ in the $o_H(1)$ term.
\end{proof}

\begin{corollary}[Sublinear corrections for cyclic odd parts]
\label{cor:cyclic-sublinear}
For every fixed odd integer $M\ge3$, as $r\to\infty$,
\begin{equation}\label{eq:cyclic-sublinear}
 \D(C_M\oplus C_2^r)=\Ds(C_M\oplus C_2^r)
 +O_M\!\left(\frac r{\log r}\right).
\end{equation}
In particular,
\[
 \frac{E_{C_{15}}(r)}{\log r}\longrightarrow\infty,
 \qquad \frac{E_{C_{15}}(r)}r\longrightarrow0.
\]
\end{corollary}
\begin{proof}
Brakensiek--Gopi--Guruswami~\cite[Corollary~4.24, preprint version]{BGG}
show that a binary linear subspace $U\le\F_2^N$ in which zero is
the only word of weight divisible by $M$ satisfies
$\dim U=O_M(N/\log(N+2))$, assuming bounded-torsion PFR.
That hypothesis is supplied by
Gowers--Green--Manners--Tao~\cite[Theorem~1.1]{GGMT}.

Let a zero-sum-free sequence of length $L$ over $C_M\oplus C_2^r$
have cyclic labels $h_i\in\{0,\ldots,M-1\}$. Its binary selection
kernel $V$ has dimension $d\ge L-r$, and
$\sum_i h_ix_i=0\pmod M$ has the unique solution $x=0$ on $V$.
Replace coordinate $x_i$ by $h_i$ copies, omitting it if $h_i=0$.
This linear map $T$ has target length $N=\sum_i h_i\le(M-1)L$
and is injective on $V$: if $Tx=0$, its weighted label sum is zero,
so $x=0$. Moreover $\wt(Tx)\equiv\sum_i h_ix_i\pmod M$.
Applying the cited code bound to $T(V)$ gives
\[
 L-r\le d=O_M\!\left(\frac N{\log(N+2)}\right)
          =O_M\!\left(\frac L{\log(L+2)}\right).
\]
This first gives $L\le2r$ for all sufficiently large $L$, and then
$L\le r+O_M(r/\log r)$. The standard lower bound proves
\eqref{eq:cyclic-sublinear}. For $M=15$, combine this upper bound
with Proposition~\ref{prop:mixed-growth} and
$\Ds(C_{15}\oplus C_2^r)=r+29$.
\end{proof}

\begin{proof}[Proof of Theorem~\ref{thm:main-threshold}]
If $t(H)=0$, then $H\oplus C_2^r$ is a two-group and Olson's
formula applies. If $t(H)=1$, Proposition~\ref{prop:D-upper}
gives the logarithmic upper bound, and Corollary~\ref{lem:simplex}
in a subgroup gives the logarithmic lower bound.
For $t(H)\ge2$, Lemma~\ref{lem:composite} applies; its lower
bound divided by $\log r$ tends to infinity.
\end{proof}

\begin{corollary}[Elementary two-primary parts]\label{cor:elementary}
For every odd prime $p$ and fixed $u\ge0$,
\[
 E_{C_2^u\oplus C_p}(r)=(p-1)\log r+O_{p,u}(1).
\]
In particular the coefficient limit exists and equals $p-1$.
\end{corollary}
\begin{proof}
Absorb $C_2^u$ into the growing binary rank. The case $Q=1$ of
Proposition~\ref{prop:D-upper} gives an upper coefficient $p-1$,
and Corollary~\ref{lem:simplex} gives the same lower coefficient.
The standard lower bound differs from the binary rank by a fixed
constant, and $\log(r+u)=\log r+O_u(1/r)$.
\end{proof}

\begin{corollary}[Nilpotent fiber threshold]\label{cor:fiber-threshold}
For a fixed finite nilpotent group $K$, the following are equivalent:
\begin{enumerate}
\item $t(K)\le1$;
\item for every fixed $\beta$, all nonempty fibers of all products
\eqref{eq:local-word} on all binary affine codes have density at
least $(T+2)^{-c_{K,\beta}}$, with an exponent independent of the
code, scopes, maps, and target.
\end{enumerate}
\end{corollary}
\begin{proof}
If $t(K)\le1$, the Sylow decomposition is $K=Q\times P$ with
$Q,P$ as in Theorem~\ref{thm:local}; either that theorem or
Corollary~\ref{cor:IKKW-density} applies.
If $t(K)=t\ge2$, the Sylow direct product contains a cyclic
subgroup of order $M$, the product of its distinct odd prime
divisors. The weighted codes in Lemma~\ref{lem:composite} give
one-bit local factors with a singleton identity fiber. Its density
is $2^{-d}$, while $d\gg B^t$ and
$\log(T+2)=O(B\log B)$. Hence $d/\log(T+2)\to\infty$,
contradicting any uniform polynomial density bound.
\end{proof}

The last equivalence concerns solution density. It does not by itself
assert computational hardness when the density bound fails.

\subsection{An explicit upper bound for \texorpdfstring{$C_{24}$}{C24}}\label{subsec:C24-upper}
For $C_{24}\cong C_8\oplus C_3$, the general coefficient in
\eqref{eq:D-upper-asymptotic} is $2^7\cdot2=256$.
A preliminary mod-four extraction reduces the binary degree
needed in the final intersection step.

\begin{lemma}[A mod-four zero subcode]\label{lem:mod4}
Let $V\le\F_2^L$ have dimension $d$, and let
$q(x)=\sum_i q_ix_i\pmod4$. There is a subcode $U\le V$
with $q|_U=0$ and
\[
 \dim U\ge\max\{0,\lfloor(d-2)/2\rfloor\}.
\]
\end{lemma}
\begin{proof}
Suppose $v_1,\ldots,v_j$ span a zero subcode. To extend it impose
\[
 q(x)=0\pmod4,\qquad
 2\sum_iq_i x_i(v_s)_i=0\pmod4\quad(1\le s\le j).
\]
These are weighted binary constraints into $C_4\oplus C_2^j$.
Their common zero indicator has binary degree at most $3+j$
by \eqref{eq:abelian-kappa} and the group-algebra expansion.
It is one at zero. Lemma~\ref{lem:binary-weight} gives at least
$2^{d-(3+j)}$ points in its support. If $d>3+2j$, this is larger than the old span,
so an extending vector exists.
The identity
\[
 q(x\oplus y)=q(x)+q(y)-2\sum_iq_ix_iy_i\pmod4
\]
shows that the enlarged span is again zero. Iteration gives the
stated dimension.
\end{proof}

\begin{lemma}[The residual cubic]\label{lem:residual-cubic}
Let $q(x)=\sum_iq_ix_i\pmod8$ on a binary linear code $U$.
If every value is divisible by four, then $q/4:U\to\F_2$
has algebraic degree at most three.
\end{lemma}
\begin{proof}
In basis coordinates $z$ on $U$, every original bit is a parity.
Its integer multilinear coefficient on a monomial of size $j$
is divisible by $2^{j-1}$. Consequently $q$ has a multilinear
representation modulo eight with no terms of degree at least four.
Since all its values are divisible by four, integer M\"obius
inversion shows that all its coefficients are divisible by four
modulo eight. Divide by four and reduce modulo two.
\end{proof}

\begin{theorem}[Explicit $C_{24}$ bounds]\label{thm:C24-upper}
For any weighted $C_{24}$-sum on a binary affine code of length
$L$ and dimension $d$, every attainable fiber has at least
\begin{equation}\label{eq:C24-density}
 \frac{2^d}{512\,\Vb{2}{L}^{16}}
\end{equation}
points. A singleton fiber implies
\begin{equation}\label{eq:C24-isolation}
 d\le9+16\log\Vb{2}{L}.
\end{equation}
For every $r\ge512$,
\begin{align}
 \D(C_{24}\oplus C_2^r)
 &\le r+10+\lfloor16\log(2r^2+r+1)\rfloor,
       \label{eq:C24-upper-finite}\\
 E_{C_{24}}(r)&\le32\log r+3.\label{eq:C24-upper-log}
\end{align}
\end{theorem}
\begin{proof}
First suppose the zero fiber on a linear code is $\{0\}$.
Apply Lemma~\ref{lem:mod4} to the reduction modulo four of its
$C_8$-sum. On a subcode $U$ of dimension
$k\ge\max\{0,\lfloor(d-2)/2\rfloor\}$, the remaining $C_8$
zero indicator is a binary polynomial of degree at most three,
by Lemma~\ref{lem:residual-cubic}. The $C_3$ zero indicator has
Fourier support at most $\Vb{2}{L}$.
Theorem~\ref{thm:main-intersection} gives
$k\le3+8\log\Vb{2}{L}$. Since $d\le2k+3$, apart from the
harmless case $d\le3$, this proves \eqref{eq:C24-isolation}.

Translate an arbitrary attained affine target to zero. Toggling
a bit $x_i$ about a fixed solution replaces its weight by
$(-1)^{(x_0)_i}h_i$ and subtracts the target. Thus the singleton
bound is stable under every affine restriction. The maximal
isolating-subspace argument in the proof of
Theorem~\ref{thm:main-intersection} gives
\eqref{eq:C24-density}.

Use the proof of Proposition~\ref{prop:D-upper} with
$A(L)=9+16\log\Vb{2}{L}$. Now
$\Vb{2}{2r}=2r^2+r+1$, and the necessary finite cutoff holds
for $r\ge512$. Indeed
$A(2r)\le41+32\log r$, and
$r-42-32\log r$ is positive at $512$ and increasing thereafter.
This proves \eqref{eq:C24-upper-finite}. Finally
\[
 16\log(2r^2+r+1)
 \le32\log r+16+16\log(1+1/r)
 <32\log r+17
\]
in this range. Since $\Ds(C_{24}\oplus C_2^r)=r+24$,
\eqref{eq:C24-upper-log} follows.
\end{proof}

\subsection{Direct comparison with a covering criterion}
\label{subsec:covering}
Han--Zhang's covering criterion~\cite[Proposition 2.4]{HZ}, applied
with odd primary part $C_3$ and complementary group
$C_8\oplus C_2^r$, requires
\begin{equation}\label{eq:HZ-C24}
 8\,2^r\sum_{j=0}^2\binom Nj7^{N-j}<8^N.
\end{equation}
At $N=r+O(\log r)$ even the $j=0$ term prevents this inequality:
its ratio to $8^N$ has binary logarithm
$3+r-N\log(8/7)$, which tends to infinity.
Thus this direct substitution does not give
\eqref{eq:C24-upper-log}. This observation concerns that
substitution, not a non-implication claim about every possible use
of the criterion. The alternative qualitative derivation through
fixed-group programs was established in Section~\ref{sec:programs}.
 
\section{Entropy lower bounds from finite configurations}\label{sec:entropy}
We now turn from fiber lower bounds to constructions of singleton
fibers. The following exact identity permits the internal
prime-power scale to grow while the odd target group remains fixed.

\subsection{An exact prime-power detector}
For $U\subseteq[n]$, write
$\ell_U(x)=\bigoplus_{i\in U}x_i\in\{0,1\}$ and
$\chi_U(x)=(-1)^{\ell_U(x)}$.

\begin{lemma}[Truncated Fourier detector]\label{lem:detector}
Let $p$ be an odd prime, $k\in\Z_{\ge1}$, and $A\in\Z_{\ge2}$.
Set $t=p^k$ and $n=At-1$.
Define $b(x)=\Id_{t\mid\wt(x)}$ on $\F_2^n$, with values in $\F_p$.
Then
\begin{equation}\label{eq:detector-fourier}
 b(x)=\sum_{\substack{U\subseteq[n]\\|U|<t}}\chi_U(x).
\end{equation}
Every displayed Fourier coefficient is exactly one; all others
are zero. Moreover
\begin{equation}\label{eq:detector-parity}
 2\sum_{\substack{U\subseteq[n]\\1\le|U|<t}}\ell_U(x)=1-b(x)
 \quad\text{in }\F_p.
\end{equation}
\end{lemma}
\begin{proof}
Fix a frequency $U$ of size $j$, and write $j=qt+s$ with
$0\le s<t$. The unnormalized Fourier coefficient of $b$ is
the sum of all coefficients whose exponent is divisible by $t$
in the polynomial
\[
 (1-z)^j(1+z)^{n-j}.
\]
In characteristic $p$, Frobenius gives the factorization
\[
 (1-z^t)^q(1+z^t)^{A-q-1}
 (1-z)^s(1+z)^{t-1-s}.
\]
The final factor has degree $t-1$ and constant term one.
Only its constant term can contribute to exponents divisible
by $t$. The sum of the desired coefficients is therefore
\[
 \left.(1-y)^q(1+y)^{A-q-1}\right|_{y=1}
 =\begin{cases}2^{A-1},&q=0,\\0,&q\ge1.\end{cases}
\]
After normalization, the nonzero coefficient is
$2^{A-1-n}=2^{A(1-t)}=1$ in $\F_p$, since $p-1\mid t-1$.
This proves \eqref{eq:detector-fourier}.

Evaluating at zero gives
$\sum_{j=1}^{t-1}\binom nj=0$ in $\F_p$.
Substitute $\chi_U=1-2\ell_U$ into
\eqref{eq:detector-fourier} to obtain
\eqref{eq:detector-parity}.
\end{proof}

The output in \eqref{eq:detector-parity} is always zero or one
in the same fixed field $\F_p$. The growing number $t=p^k$ is
a weight-divisibility scale, not the order of the target group.

\subsection{The configuration transfer theorem}
Let $Q$ be a finite abelian two-group. A finite list
$(q_i,a_i)_{i=1}^s$, with $q_i\in Q$ and $a_i\in\Z_{\ge1}$,
will be called an admissible configuration if, for all
$(z_1,\ldots,z_s)\in\Z^s$,
\begin{equation}\label{eq:box-condition}
 0\le z_i\le a_i\ (1\le i\le s),\quad
 \sum_i z_iq_i=0
 \quad\Longrightarrow\quad z_1=\cdots=z_s=0.
\end{equation}
Equivalently, the multiset with $a_i$ copies of each $q_i$ is
zero-sum-free. The same label may occur in more than one block.

\begin{theorem}[Transfer from a finite configuration]\label{thm:box}
Let $P=C_p^v$, where $p$ is an odd prime and $v\in\Z_{\ge1}$, and suppose
$(q_i,a_i)_{i=1}^s$ is admissible in $Q$, with
$1\le s\le v(p-1)$. For each $k\in\Z_{\ge1}$, put
\[
 A_i=a_i+1,\qquad t=p^k,\qquad n_i=A_it-1,
 \qquad d_k=\sum_i n_i,
 \qquad R_k=\sum_{i=1}^s\sum_{j=1}^{t-1}\binom{n_i}{j}.
\]
There is a zero-sum-free sequence of length $R_k+d_k$ over
$Q\oplus P\oplus C_2^{R_k}$. Consequently
\begin{equation}\label{eq:box-entropy}
 \overline\lambda(Q\oplus P)\ge
 \frac{\sum_i A_i}{\max_i\{A_iH_2(1/A_i)\}}.
\end{equation}
\end{theorem}
\begin{proof}
Choose $h_1,\ldots,h_s\in P$ forming a zero-sum-free sequence.
This is possible by taking a subsequence of the standard sequence
with $p-1$ copies of each of $v$ independent generators.

Take independent binary blocks $x^{(i)}\in\F_2^{n_i}$.
For each $i$ and every nonempty $U\subseteq[n_i]$ with $|U|<t$,
introduce a parity coordinate
$y_{i,U}=\ell_U(x^{(i)})$. The resulting graph code
\[
 V_k=\{(x,Tx):x\in\F_2^{d_k}\}\subseteq\F_2^{d_k+R_k}
\]
has full-rank parity-check matrix $B=[\,T\mid I_{R_k}\,]$.
Give a systematic coordinate from block $i$ the label $(q_i,0)$
in $Q\oplus P$, and a parity coordinate from block $i$ the label
$(0,2h_i)$. Attach the corresponding column of $B$ to each
label, obtaining a sequence in $Q\oplus P\oplus C_2^{R_k}$.

If a selection has zero binary sum, its parity bits are $y=Tx$.
Its odd component is, by Lemma~\ref{lem:detector},
\[
 \sum_i h_i\bigl(1-b_i(x^{(i)})\bigr),
 \qquad b_i(x^{(i)})=\Id_{t\mid\wt(x^{(i)})}.
\]
Each coefficient is zero or one. Since the sequence $(h_i)$ is
zero-sum-free, vanishing forces every coefficient to be zero.
Thus $\wt(x^{(i)})=tz_i$ with $0\le z_i\le a_i$.
The $Q$ component is $t\sum_i z_iq_i$. Multiplication by the
odd integer $t$ is invertible on $Q$, so
\eqref{eq:box-condition} gives $z_i=0$ for every $i$.
Hence $x=0$ and then $y=0$. The sequence is zero-sum-free.

For a fixed integer $A\ge2$, Stirling's formula gives
\begin{equation}\label{eq:entropy-binomial}
 \log\left(\sum_{j=1}^{t-1}\binom{At-1}{j}\right)
       =At\,H_2(1/A)+o(t).
\end{equation}
For $A>2$, bound the sum between its largest term and $At$
times that term and apply Stirling's formula. For $A=2$, the
sum equals $2^{2t-2}-1$, giving the same estimate.
Since the number of blocks is fixed,
\[
 d_k=t\sum_i A_i+O(1),\qquad
 \log R_k=t\max_i\{A_iH_2(1/A_i)\}+o(t).
\]
The sequence length gives
$\D(Q\oplus P\oplus C_2^{R_k})\ge R_k+d_k+1$.
Subtract $\Ds=R_k+O_{Q,P}(1)$ and divide by $\log R_k$.
This proves \eqref{eq:box-entropy}.
\end{proof}

\begin{definition}[Finite entropy bound]\label{def:box-invariant}
For a positive integer $b$, define
\[
 \mathcal B_b(Q)=\max\left(
 \{0\}\cup
 \left\{
 \frac{\sum_{i=1}^s(a_i+1)}
 {\max_i\{(a_i+1)H_2(1/(a_i+1))\}}:
 \begin{array}{l}
 1\le s\le b,\\[-2pt]
 (q_i,a_i)_{i=1}^s\text{ admissible in }Q
 \end{array}
 \right\}\right).
\]
\end{definition}

This is a finite maximum: every admissible list satisfies
$\sum_i a_i\le\D(Q)-1$, and the labels are chosen in a fixed
finite group. Theorem~\ref{thm:box} gives
\begin{equation}\label{eq:finite-entropy-bound}
 \overline\lambda(Q\oplus C_p^v)\ge\mathcal B_{v(p-1)}(Q).
\end{equation}
This parameter packages a family of lower bounds; it is not asserted
to equal the Davenport coefficient.

\begin{example}[Elementary two-groups]
For $Q=C_2^u$, every admissible list has $a_i=1$ and linearly
independent labels. Thus
$\mathcal B_b(C_2^u)=\min\{u,b\}$.
On the other hand, the weighted odd-primary bound and $v$ independent
simplex constructions give
\[
 E_{C_2^u\oplus C_p^v}(r)=v(p-1)\log r+O_{p,u,v}(1).
\]
For the lower bound, use $p-1$ simplex blocks in each independent
$C_p$ coordinate; their total dimension is $v(p-1)m$, and
$m=\log r+O_{p,v}(1)$. For the upper bound absorb $C_2^u$ into
the binary rank and use $b_P=v(p-1)$ in
Corollary~\ref{cor:weighted}.
In particular \eqref{eq:finite-entropy-bound} need not be sharp.
\end{example}

\subsection{The \texorpdfstring{$C_{24}$}{C24} family and unbounded coefficient growth}
\begin{corollary}[An explicit $C_{24}$ lower family]\label{cor:C24-lower}
For each integer $k\ge1$, put $t=3^k$ and
\[
 n_1=5t-1,\qquad n_2=4t-1,\qquad
 R_k=\sum_{j=1}^{t-1}
       \left(\binom{n_1}{j}+\binom{n_2}{j}\right).
\]
Then
\begin{equation}\label{eq:C24-lower-finite}
 E_{C_{24}}(R_k)\ge9\cdot3^k-25,
\end{equation}
and
\begin{equation}\label{eq:C24-lower-coefficient}
 \overline\lambda(C_{24})\ge\frac{9}{5H_2(1/5)}.
\end{equation}
\end{corollary}
\begin{proof}
In $Q=C_8$, let $q$ be a generator and take two blocks with
$(q_1,q_2)=(q,q)$ and $(a_1,a_2)=(4,3)$.
If $0\le z_1\le4$ and $0\le z_2\le3$, then
$0\le z_1+z_2\le7$. Hence the configuration is admissible.
Take $P=C_3$ and $h_1=h_2$ a generator.
Theorem~\ref{thm:box} gives a zero-sum-free sequence of length
$R_k+9t-2$. Since $\Ds(C_{24}\oplus C_2^{R_k})=R_k+24$,
this proves \eqref{eq:C24-lower-finite}.
The larger of the two entropy exponents is $5H_2(1/5)$,
which proves \eqref{eq:C24-lower-coefficient}.
\end{proof}

Here is the sequence without matrix notation. Let $g$ have order
24, and let $e_{i,U}$, with $i=1,2$ and
$\varnothing\ne U\subseteq[n_i]$, $|U|<t$, be independent
elements of order two. Take the terms
\begin{equation}\label{eq:C24-explicit-sequence}
 9g+\sum_{U\ni j}e_{i,U}\quad(1\le j\le n_i),
 \qquad
 8g+e_{i,U}\quad(1\le|U|<t),
 \qquad i=1,2.
\end{equation}
The residues $9$ and $8$ correspond to $(1,0)$ and $(0,2)$
under $C_{24}\cong C_8\oplus C_3$.
If the binary sum vanishes, the second type of term is selected
exactly according to the relevant parity of the first type.
The odd component forces the two block weights to have the form
$w_1=tz_1$ and $w_2=tz_2$, with $0\le z_1\le4$ and $0\le z_2\le3$.
The $C_8$ component gives $t(z_1+z_2)\equiv0\pmod8$.
Since $t=3^k$ is invertible modulo eight, $z_1+z_2\equiv0\pmod8$.
Now $0\le z_1+z_2\le7$ forces $z_1=z_2=0$, and hence both weights vanish.

For orientation, the finite lower estimates are:
\[
\begin{array}{c|r|r|r|r}
 k&d_k&\log_2R_k&9\cdot3^k-25&
 (9\cdot3^k-25)/\log_2R_k\\ \hline
 2&79&27.932222&56&2.004853\\
 3&241&91.980460&218&2.370069\\
 4&727&286.127990&704&2.460437\\
 5&2185&870.104755&2162&2.484758
\end{array}
\]
These entries illustrate the explicit family; its asymptotic
coefficient follows from \eqref{eq:entropy-binomial}.

\begin{corollary}[Growing two-primary exponent]\label{cor:unbounded}
For every odd prime $p$ and integer $a\ge1$,
\begin{equation}\label{eq:unbounded-coefficient}
 \overline\lambda(C_{p2^a})\ge\frac1{H_2(2^{-a})}.
\end{equation}
\end{corollary}
\begin{proof}
Take one block in $Q=C_{2^a}$ with a generator $q_1$ and
$a_1=2^a-1$. This is admissible, and
Theorem~\ref{thm:box} with $P=C_p$ gives
\eqref{eq:unbounded-coefficient}.
The expansion
$H_2(\varepsilon)=\varepsilon\log_2(1/\varepsilon)
+\varepsilon\log_2e+O(\varepsilon^2)$ gives
\eqref{eq:main-unbounded}.
\end{proof}

\begin{proof}[Proof of Theorem~\ref{thm:main-coefficients}]
Combine Corollary~\ref{cor:C24-lower} with
Theorem~\ref{thm:C24-upper} for \eqref{eq:main-C24}.
Corollary~\ref{cor:unbounded} proves
\eqref{eq:main-unbounded}.
\end{proof}

Together with Corollary~\ref{cor:elementary}, the $C_{24}$ example
rules out a coefficient formula depending only on the odd primary
part. In particular no constant depending on $C_8$ and $C_3$ can
make $E_{C_{24}}(r)=2\log_2r+O(1)$.

\section{Limitations and further questions}\label{sec:questions}
For $C_3\oplus C_2^r$, Theorem~\ref{thm:windows} gives numerical
equalities on infinitely many explicit rank intervals.
Theorem~\ref{thm:large-rank} gives a structural identity at every
sufficiently large rank, with two Fourier thresholds still to be
evaluated. These are different levels of determination.
For general odd $p$-groups, Theorem~\ref{thm:dichotomy} identifies
the logarithmic coefficient $b=\D(H)-1$, but an exact bounded remainder must
retain the discrete dependence on $r$. Indeed, if
$\D(H\oplus C_2^r)-\Ds(H\oplus C_2^r)-b\log_2r$ converged,
its limit along $r=2^j$ would be an integer, whereas along
$r=\lfloor2^{j+1/(2b)}\rfloor$ it would be a half-integer.

\subsection{A necessary loss in degree--sparsity intersection bounds}
\label{subsec:obstruction}
Theorem~\ref{thm:main-intersection} has a large dependence on $D$.
Some dependence on the binary degree is unavoidable even with a
fixed odd output field.

\begin{proposition}[A cubic obstruction]\label{prop:obstruction}
There is no constant $C$ such that every pair
$a:\F_2^d\to\F_2$, $b:\F_2^d\to\F_3$ with
$\deg a\le3$ and $\{a=1,b\ne0\}=\{0\}$ satisfies
$d\le\log_2\Sp(b)+C$. This remains false if $b$ is restricted
to the values zero and one.
\end{proposition}
\begin{proof}
Take $t=3^k$, $d=4t-1$, and let
$b(x)=\Id_{t\mid\wt(x)}$. Define the binary polynomial
$a(x)=(1+e_1(x))(1+e_2(x))$.
The functions $e_1,e_2$ give the lowest two binary digits of
$\wt(x)$, so $a(x)=1$ exactly when $4\mid\wt(x)$, and
$\deg a\le3$.
The common condition implies $4t\mid\wt(x)$ and therefore
$x=0$. Lemma~\ref{lem:detector} gives
\[
 \Sp(b)=\sum_{j=0}^{t-1}\binom dj,
 \qquad \frac{\log_2\Sp(b)}d\longrightarrow H_2(1/4)<1.
\]
Thus $d-\log_2\Sp(b)\to\infty$.
\end{proof}

More generally, taking the indicator of weight divisible by $2^a$
gives binary degree at most $2^a-1$ and sparsity exponent
$H_2(2^{-a})$. Thus any universal singleton estimate
\[
 d\le C_D\log_2\Sp(b)+O_D(1)
\]
must satisfy
\[
 C_{2^a-1}\ge\frac1{H_2(2^{-a})}\qquad(a\ge2).
\]
The upper coefficient $2^D$ and this necessary lower dependence
leave a substantial gap. Its study is independent of whether a
specific group-valued fiber has extra structure.

\subsection{Coefficient limits and optimal constructions}
For fixed $H$ with $t(H)=1$, the threshold theorem places
$\underline\lambda(H)$ and $\overline\lambda(H)$ between positive
finite constants. It does not show that they are equal.
Theorem~\ref{thm:box} supplies families along the prime-power
scales $t=p^k$; no assertion of a single coefficient limit is
needed for these lower bounds.

Three questions arise naturally. First, does the limit
$E_H(r)/\log_2r$ exist for every such $H$? Second, how closely
does the finite configuration bound
$\mathcal B_{v(p-1)}(Q)$ approximate the upper limiting coefficient
for $H=Q\oplus C_p^v$? Third, can the degree--sparsity intersection
bound incorporate more of the two-primary algebra and the odd
weight detector, leading to matching upper bounds?

The first concrete case is $C_{24}$, where the present results give
\[
 2\le\underline\lambda(C_{24})\le\overline\lambda(C_{24})\le32,
 \qquad
 \overline\lambda(C_{24})\ge2.4933231062\ldots.
\]
The lower bound two for the liminf comes from the simplex
construction. The larger limsup bound proves that the formula
with coefficient two is false, without deciding whether the
normalized excess converges.

\subsection{More than one odd prime}
When $t(H)\ge2$, the composite-modulus construction proves
superlogarithmic growth. Corollary~\ref{cor:cyclic-sublinear} gives
a sublinear upper bound for cyclic odd $H$, but matching bounds
remain open. The fixed-group program used in
Proposition~\ref{prop:wreath} then has a normal base involving
several odd primes, so the specific class
\eqref{eq:IKKW-class} no longer applies through that decomposition.
The elementary intersection theorem also controls only one
odd-characteristic Fourier condition at a time. A quantitative
principle for several odd characteristics would need to account
for the composite-modulus singleton constructions.

\appendix
\section{A transfer to nonabelian Davenport constants}\label{sec:nonabelian}
For a finite group $G$, let $\dd(G)$ be the maximum length of a
product-one-free sequence, where selected terms may be reordered.
Let $\D(G)$ be the maximum length of an atom in the monoid of
product-one sequences. In the abelian case $\D(G)=\dd(G)+1$.

\begin{proposition}\label{prop:nonabelian}
Let $K$ be any fixed finite group whose prime divisors lie in
$\{2,p\}$, with $p$ odd. Then
\[
 \dd(K\times C_2^r)=r+O_K(\log(r+2)),\qquad
 \D(K\times C_2^r)=r+O_K(\log(r+2)).
\]
If $p\mid|K|$, both corrections to $r$ are
$\Theta_K(\log r)$. No nilpotence assumption on $K$ is needed.
\end{proposition}
\begin{proof}
Define the fixed abelian certificate group
\[
 B_K=\bigoplus_{g\in K\setminus\{1\}}C_{\ord(g)}.
\]
Replace the $K$-label $g$ of a term by the generator of its own
cyclic summand of $B_K$; give the identity label certificate zero.
A zero-sum subsequence of the certificates uses each nonidentity
$g$ a multiple of $\ord(g)$ times. Its $K$ terms can therefore
be grouped into identity blocks. If its binary sum vanishes,
it is product-one in $K\times C_2^r$.
Consequently
\[
 \dd(K\times C_2^r)+1\le\D(B_K\oplus C_2^r).
\]
Proposition~\ref{prop:D-upper} bounds the right-hand side by
$r+O_K(\log(r+2))$. The trivial lower bound for the small
constant is $r$, and the subgroup $C_p\le K$, when present,
gives the logarithmic lower bound by Corollary~\ref{lem:simplex}.

For the large constant, the general inequality
\[
 \dd(G)+1\le\D(G)\le\dd(G)+2|G'|-1
\]
uses the upper bound of~\cite[Theorem 3.1]{Grynkiewicz}.
Since $(K\times C_2^r)'=K'$, the gap is bounded independently
of $r$, proving the remaining assertions.
\end{proof}

\section{An exact structure formula for all sufficiently large ranks}\label{sec:large-rank}

The exact transfer can be localized to two adjacent dimensions. The
construction needed for this localization is uniform in all sufficiently
large even dimensions. The result below retains two unevaluated
extremal thresholds; it is not a closed evaluation for all ranks.

\begin{theorem}[Two-threshold localization]
\label{thm:large-rank}
Let $r\ge15\cdot2^{63}-133$ be an integer, and put
\[
 m=\floor{\log_2\!\left(\frac8{15}(r+2\log_2r)\right)}.
\]
Then
\begin{equation}\label{eq:large-rank}
\begin{aligned}
 \D(C_3\oplus C_2^r)
 ={}&r+2m+1\\
 &+\ind_{\{\nu(2m+1)\le r+2m+2\}}\\
 &+\ind_{\{\nu(2m+2)\le r+2m+3\}}.
\end{aligned}
\end{equation}
\end{theorem}

The indicators retain two unevaluated extremal thresholds. Their sum
is zero, one, or two, and each determines the exact feasible increment.
The sufficient cutoff is explicit and is not optimized.
The proof uses a construction valid in every sufficiently large even
dimension.

\begin{theorem}[Uniform even-dimensional construction]\label{thm:even-construction}
For every integer $M\ge66$, there is a function
$f:\F_2^{2M}\to\F_3$ whose zero set is $\{0\}$ and whose nonzero Fourier
coefficients are all equal to one, with
\[
 |\supp\widehat f|=15\cdot2^{M-3}.
\]
In particular, $\nu(2M)\le15\cdot2^{M-3}$.
\end{theorem}

\subsection{A construction from finite-field lines}
For a finite subset $T$ of a binary vector space, put
\[
 U_T(x)=\sum_{t\in T}(-1)^{\langle t,x\rangle}\quad\hbox{in }\F_3.
\]
We will use finite sets $T$ satisfying
\begin{equation}\label{eq:nonvanishing}
 U_T(x)\ne0\qquad\text{for every character }x.
\end{equation}
Singletons satisfy \eqref{eq:nonvanishing}, as do translates of
$W\setminus\{0\}$ when $W$ has odd dimension. Indeed, the latter
character sums, up to sign, are $2^{\dim W}-1$ and $-1$, both nonzero
modulo three. Cartesian products and translations preserve
\eqref{eq:nonvanishing}.

\begin{lemma}[Extension using finite-field lines]\label{lem:spread-extension}
Let $j\in\Z_{\ge1}$ and $q=2^j$. Suppose $S\subseteq V$ has
$U_S^{-1}(0)=\{0\}$ and admits a disjoint
partition into $q+1$ sets $T_0,\ldots,T_q$ satisfying
\eqref{eq:nonvanishing}.
Then there is such a support in dimension $\dim V+2j$ of size $q|S|$.
\end{lemma}
\begin{proof}
Let $L_0,\ldots,L_q$ be the one-dimensional $\F_q$-subspaces of
$\F_q^2$, viewed as binary $j$-spaces, and set
\[
 S'=\bigsqcup_{i=0}^q(T_i\times L_i).
\]
The annihilators $L_i^\perp$ also partition the nonzero dual vectors.
Indeed, distinct $L_i$ span $\F_q^2$, so their annihilators meet only
at zero; counting $(q+1)(q-1)=q^2-1$ nonzero vectors proves the claim.
At a character $(x,0)$ the sum is $qU_S(x)$. At $(x,y)$ with $y\ne0$,
exactly one annihilator contains $y$, and the sum is $qU_{T_i}(x)$.
Since $q$ is invertible in $\F_3$, the unique zero is $(0,0)$.
The partition is disjoint because the $T_i$ are disjoint.
\end{proof}

\subsection{An eight-dimensional example}
Take independent binary three-spaces $X,Y$, and a binary two-space $Z$
with distinct nonzero vectors $\ell_0,\ell_1,\ell_2$. In $X\oplus Y\oplus Z$
put
\[
 S=\{0,\ell_0\}
 \sqcup\{x,x+\ell_1:0\ne x\in X\}
 \sqcup\{y,y+\ell_2:0\ne y\in Y\}.
\]
This set has thirty points, and $U_S$ vanishes only at zero: apply
Lemma~\ref{lem:spread-extension} with $q=2$ to
$X\cup Y$, partitioned as $\{0\},X\setminus0,Y\setminus0$.
The character sum of $X\cup Y$ is
$8\ind_{x=0}+8\ind_{y=0}-1$, which has precisely one zero modulo three.

The sets $Y\setminus0$ and $(Y\setminus0)+\ell_2$ are disjoint
seven-point sets in $S$ satisfying \eqref{eq:nonvanishing}. Let $V_X=X\oplus\langle\ell_1\rangle$.
For a basis $x_1,x_2,x_3$ of $X$, define
\[
 E=(\langle x_1,x_2\rangle\setminus0)
    \sqcup(\langle x_3,\ell_1\rangle\setminus0).
\]
The character sum of this six-point set vanishes only at zero in $V_X$, since it equals
$4\ind_{\xi=0}+4\ind_{\eta=0}-2$.
It contains $\ell_1$, whereas $S\cap V_X=V_X\setminus\{\ell_1\}$.
Thus $T=V_X\setminus E$ is a ten-point subset of $S$.
It satisfies \eqref{eq:nonvanishing}: at the trivial restricted character its sum is
$16-6=10$, and elsewhere its sum is $-U_E\ne0$.
The two seven-point sets and $T$ leave six singleton points. We have
therefore partitioned $S$ into nine parts, each satisfying
\eqref{eq:nonvanishing}, of sizes
\[
 10,\ 7,\ 7,\ 1,\ 1,\ 1,\ 1,\ 1,\ 1.
\]

\subsection{Refining the partition and covering all large even dimensions}
Suppose $q=2^j$, with $j\ge3$ odd, and $U_S^{-1}(0)=\{0\}$.
Assume $S$ has a partition into $q+1$ parts satisfying
\eqref{eq:nonvanishing}, including distinct parts $T_a,T_b,T_c$ of sizes
seven, seven, and ten.
The extension in Lemma~\ref{lem:spread-extension} admits partitions into
both $4q+1$ and $16q+1$ parts satisfying \eqref{eq:nonvanishing}.
The three parts $T_a\times\{0\},T_b\times\{0\},T_c\times\{0\}$
remain in each partition.

For the first transition, initially partition each product into
\[
 T_i\times\{0\},\qquad T_i\times(L_i\setminus0).
\]
These satisfy \eqref{eq:nonvanishing} because $j$ is odd, and there are
$2(q+1)$ pieces.
Choose a three-space $W\subseteq L_a$ and replace the nonzero layer of
$T_a$ by $T_a\times(W\setminus0)$ and the $q-8$ translates
$T_a\times\{z\}$ for $z\in L_a\setminus W$; this adds $q-8$ pieces.
Replace the nonzero layer of $T_c$ by its $q-1$ translates, adding
$q-2$ pieces, and split one ten-point translate into singletons,
adding nine. The total is
\[
 2(q+1)+(q-8)+(q-2)+9=4q+1.
\]

For the second transition, start with the same $2(q+1)$ pieces.
Refine $T_i\times(L_i\setminus0)$ for $i=a,b,c$ into translates, obtaining
$5q-4$ pieces. There are $2(q-1)$ seven-point translates and $q-1$
ten-point translates. Split
\[
 N_7=\frac{11q-4}{6}
\]
of the seven-point translates and one ten-point translate into
singletons. Here $N_7$ is an integer because $q\equiv2\pmod6$, and
$0\le N_7\le2(q-1)$ because $q\ge8$. The resulting count is
\[
 5q-4+6N_7+9=16q+1.
\]
In both cases, $T_a\times\{0\},T_b\times\{0\},T_c\times\{0\}$
remain intact.
Every operation preserves \eqref{eq:nonvanishing} for each part and
preserves disjointness.

Starting with the eight-dimensional example and $j_0=3$, the next extension
exponent can consequently be chosen as $j_i+2$ or $j_i+4$.
After $k\ge1$ extensions the attainable half-dimensions are exactly
\[
 \left\{M:k^2+2k+4\le M\le2k^2+k+4,\quad M\equiv k\pmod2\right\}.
\]
Indeed,
\[
 M=k^2+2k+4+
 2\sum_{i=0}^{k-2}(k-1-i)\epsilon_i,\qquad \epsilon_i\in\{0,1\},
\]
and the subset sums of $1,\ldots,k-1$ fill the interval from zero to
$k(k-1)/2$. The step-two intervals for $k$ and $k+2$ meet once
\[
 (2k^2+k+4)+2-\bigl((k+2)^2+2(k+2)+4\bigr)
 =(k-6)(k+1)\ge0.
\]
They cover every even $M\ge52$ and every odd $M\ge67$, hence every
$M\ge66$. Each extension multiplies the support size by $2^j$ and
increases the half-dimension by $j$. The final size is therefore
$30\cdot2^{M-4}=15\cdot2^{M-3}$.
This proves Theorem~\ref{thm:even-construction}.

\subsection{Proof of the large-rank formula}
We first record the monotonicity used to pass from feasibility of a
dimension to feasibility of the preceding one.
For $d\ge2$, choose $f$ attaining $\nu(d)$. Then $f^2=1-\delta_0$.
Every nonconstant Fourier coefficient of this square equals
$-2^{-d}\ne0$, so the support of $\widehat f$ contains a pair differing
by any prescribed nonzero direction $h$.
Upon restriction to $h^\perp$, the coefficients indexed by $u$ and
$u+h$ are replaced by their sum.
At least one pair of nonzero coefficients is thereby replaced by
at most one nonzero coefficient. The restricted function still has
zero set $\{0\}$. Hence
\[
 \nu(d-1)\le\nu(d)-1,\qquad
 t_{d-1}\le t_d,\qquad t_d=\nu(d)-d-1.
\]

Put
\[
 A_m=\frac{15}{8}2^m-2m-1.
\]
The sequence $A_m$ is strictly increasing for $m\ge5$.
For integer $r\ge A_{66}$ the unique integer $m\ge66$ with
$A_m\le r<A_{m+1}$ is
\[
 m=\floor{\log_2\!\left(\frac8{15}(r+2\log_2r)\right)}.
\]
Here is a direct check of the rounding. The function
$r+2\log_2r$ is increasing. For $m\ge5$,
$A_m\ge(49/32)2^m>\sqrt2\,2^m$, so
$A_m+2\log_2A_m\ge(15/8)2^m$.
On the other side, $A_m-1<2^{m+1}$ implies
\[
 A_m-1+2\log_2(A_m-1)<(15/8)2^m.
\]
These two inequalities identify every integer boundary exactly.

Theorem~\ref{thm:even-construction} makes dimension $2m$ feasible in
\eqref{eq:C3transfer}. Dimension $2m+3$ is impossible: Corollary~\ref{cor:odd-support}, with
$s=m+1$, gives
\[
 \nu(2m+3)>\frac8{\sqrt3}2^m,
\]
whereas integer $r<A_{m+1}$ implies
$r+2m+4\le(15/4)2^m$.
Monotonicity excludes all higher dimensions.
The feasible maximum thus belongs to $\{2m,2m+1,2m+2\}$, and the last
dimension is feasible only if the preceding one is feasible.
Consequently
\[
 \D(C_3\oplus C_2^r)
 =r+2m+1
 +\ind_{\{\nu(2m+1)\le r+2m+2\}}
 +\ind_{\{\nu(2m+2)\le r+2m+3\}}.
\]
This proves Theorem~\ref{thm:large-rank}.
The two indicators retain the unevaluated extremal thresholds; the
formula is an exact structural identity, and the construction does not
assert equality in $\nu(2M)\le15\cdot2^{M-3}$.

\end{document}